\documentclass[english,letterpaper]{article}
\usepackage[english]{babel}
\usepackage{amsmath,amsthm,amssymb}

\usepackage{latexsym}
\usepackage{enumerate}
\usepackage{xcolor}

\newtheorem{theorem}{Theorem}
\newtheorem{lemma}{Lemma}
\newtheorem{proposition}{Proposition}
\newtheorem{corollary}{Corollary}
\newtheorem{definition}{Definition}
\theoremstyle{remark}
\newtheorem{remark}{Remark}
\newtheorem{example}{Example}

\definecolor{dark-gray}{gray}{0.25} 
\definecolor{dark-blue}{RGB}{0,0,139}

\newcommand{\ms}[1]{{\color{dark-blue} #1}}

\title{Invariant probabilities for discrete time linear dynamics via thermodynamic formalism}

\author{Artur O. Lopes\thanks{IME - UFRGS. Partially supported by CNPq}\;, Ali Messaoudi\thanks{MAT - UNESP. Partially supported by CNPq, project   311018/2018-1, and FAPESP, project 2013/24541-0}\;, Manuel Stadlbauer\thanks{IM - UFRJ. Partially Supported by CNPq, project 312632/2018-5} \; and \\  Victor Vargas\thanks{IME - UFRGS. Supported by INCTMat-CAPES grant}\;.  
}

\begin{document}

\maketitle

{\bf Abstract:}  We show the existence of invariant ergodic $\sigma$-additive probability measures with full support on $X$ for a class of linear operators 
$L: X \to X$, where $L$ is a weighted shift operator and $X$ either is the Banach space $c_0(\mathbb{R})$ or $l^p(\mathbb{R})$ for $1\leq p<\infty$. In order to do so, we adapt ideas from Thermodynamic Formalism as follows. For a given bounded Hölder continuous potential $A:X \to \mathbb{R}$, we define a transfer operator $\mathcal{L}_A$ which acts on continuous functions on $X$ and  prove that this operator satisfies a Ruelle-Perron-Frobenius theorem. That is, we show the existence of an eigenfunction for $\mathcal{L}_A$ which provides us with a normalized potential $\overline{A}$ and an action of the dual operator $\mathcal{L}_{\overline{A}}^*$ on the $1$-Wasserstein space of probabilities on $X$ with a unique fixed point, to which we refer to as Gibbs probability.
 It is worth noting that the definition of $\mathcal{L}_A$ requires an  {\it a priori} probability  on the kernel of $L$. These results are extended to a wide class of operators with a non-trivial kernel defined on separable Banach spaces.

\vspace{2mm}


\vspace*{2mm}


\vspace*{2mm}

\section{Introduction}

The main goal in here is to show existence of invariant $\sigma$-additive probabilities for linear dynamical systems in discrete time for a certain class of 
linear operators for which the formalism of the Ruelle operator is applicable. That is, we obtain our main results by adapting the formalism previously considered for non-compact generalized $XY$ models where the dynamics are given by the shift (see \cite{LMMS}, \cite{LV} and \cite{CSS}). In here,  the dynamics are obtained from a linear operator $L$ but as in the case of the $XY$-model, we require the choice of an {\it a priori} probability.

We point out that the existence of invariant probabilities for $ l^p(\mathbb{R}),\; 1 < p < \infty$,    already is known, but our methods are quite different from the standard approach.  We believe the connection with thermodynamic formalism is interesting in itself, because - among other reasons -  allows an extension of the previously known results beyond the case of linear operators defined on reflexive Banach spaces.

The use of the Ruelle operator (also called transfer operator) as a tool for the construction of Gibbs measures - which are, in particular, ergodic invariant measures with full support - is a usual approach in the area of thermodynamic formalism. Moreover, there are several works in which these properties  are studied  for a wide variety of both compact and non-compact dynamical contexts (see for instance \cite{BCLMS, LV, MU, CSS, Sar}). Thus, it is natural to choose this tool in order  to solve an important question  in linear dynamics: the existence of probability measures satisfying the above mentioned properties.

Linear dynamics is a relatively young branch of mathematics in which dynamical properties of linear operators defined on Fr\'echet spaces are studied and, in particular, some interesting phenomena even in the case of Banach spaces appear. In the finite dimensional case, it is widely known that linear dynamical systems are completely characterized by their corresponding Jordan canonical form. However, in infinite dimension, interesting properties such as the existence of dense orbits, chaos in the sense of Devaney and topological equivalence of invariant subsets with any dynamical system defined on a compact metric space might occur. For an interesting example of a linear operator defined on a Hilbert space, see for instance \cite{BaMa, GrMa}.

Let $X$ be a topological vector space and $T: X \to X$ be a linear continuous operator. We say that $(X,T)$ - or simply the map $T$ - is {\it uniform hypercyclic}, if it has a dense orbit in $X$. In the particular case that $X$ is a separable Banach space, this property is equivalent to say that $T$ is {\it topologically transitive}, that is, for all non-empty open sets $U, V \subset X$, there is an integer $n \geq 0$, such that, $T^{n}(U) \cap V \neq \emptyset$. Moreover, we call the map $T$ {\it frequently hypercyclic}, if for every non-empty open set $V \subset X$, the set $N(x, V) = \{k \in \mathbb{N},\; T^{k}(x) \in V \}$ has positive lower density, i.e $ \liminf_{n \to \infty} \; \frac{1}{n} \# \left(N(x, V) \cap \{1, ..., n\}\right) > 0 \;$.
On the other hand, we call the map $T$ {\it Devaney chaotic}, if it is topologically transitive and has a dense subset of periodic points. We say that $T$ is {\it topologically mixing}, if for all non-empty open subsets $U, V \subset X$, there exists an integer $N > 0$, such that, $T^{n}(U) \cap V$ is not empty for all $n \geq N$.

A typical example of mixing, frequently hypercyclic and Devaney chaotic operator is   $\alpha L,\; \alpha > 1$, where $L$ is the shift operator acting on the Banach space $X =  l^p(\mathbb{R}),\; 1 \leq p < \infty$, i.e. $L((x_n)_{n\geq1}) = (x_{n+1})_{n\geq1}$. Furthermore, the operator $\alpha L,\; \alpha > 1$, acting  on $c_0 (\mathbb{R})$ is mixing, frequently hypercyclic and Devaney chaotic as well.
The study of the above  properties is a central problem in the area of discrete time  Linear Dynamical Systems (see for instance \cite{BaMa}, \cite{BeMe}, \cite{GrMa}, \cite{BCM}).


In the context of ergodic theory on linear dynamics, there are some results related to the existence of invariant probability measures when the system is frequently hypercyclic and $X$ is a reflexive Banach space (see for example \cite{GriMa}). We point out that in there, the existence of ergodic measures with full support are obtained through approximations of measures supported on dense orbits of the system. Here we prove a similar kind of result (even for non-reflexive Banach spaces as $l^1(\mathbb{R})$ and $c_0(\mathbb{R})$) in the  context of weighted shifts using a different technique: we adapt classical tools of thermodynamic formalism to extend this result to a large class of linear operators defined on separable and non-necessarily reflexive Banach spaces.

Let $X$ be one of the Banach spaces $c_0(\mathbb{R})$ or $l^p ( \mathbb{R})$, for $ 1 \leq p < \infty$, where $c_0(\mathbb{R})$ refers to the Banach space of sequences  
$(x_n)_{n \in \mathbb{N}}$,  with $\lim_{n \to \infty}  x_n= 0$ equipped with the $\sup$-norm. For a fixed operator $T: X \to X$, a H\"older continuous potential $A: X \to \mathbb{R}$ and a suitable {\it a priori} probability measure $m$ defined on the kernel of $T$, we are able to associate a \emph{Gibbs probability measure} for the potential $A$  (see Definition \ref{Ro} and item c) in Theorem \ref{RPF}) which will be $T$-invariant, ergodic and has full support.

Our first result is Theorem \ref{RPF}.  From this we will obtain a $\sigma$-additive ergodic probability measure which is invariant for the weighted shift $L$ and has full support. In Proposition \ref{RPF-separable} and Corollary \ref{corollary-RPF-separable}, we then extend the previous results to a large class of linear operators defined on separable Banach spaces, among them the class of uniform hypercyclic linear operators. Furthermore, these results are satisfied for the class of linear continuous operators with non-trivial kernel defined on separable Banach spaces.
An important issue here is the use of ideas of transport theory which are employed in the Appendix (Section \ref{ap})
in order to prove item b) of Theorem 1, that is the existence of a fixed point of the dual of the Ruelle operator through a contraction argument (for related results, see  \cite{Sta, BS,KLS, K, Lo, Elis}). However, the arguments here are more complex due to the   metric structure of $X$. For example, in contrast to shift spaces (even with respect to uncountable, non-compact alphabets as in \cite{CSS}), it seems to be impossible to obtain uniform contraction rates (see Remark \ref{esta1}).    

The paper is organized as follows. In Section \ref{preliminaries}, we consider some definitions and preliminaries. In Section \ref{RPF-theorem}, we prove the main result for weighted shifts operators. In Section \ref{extension}, we consider the general case of frequently hypercyclic operators  defined on separable Banach spaces. In the Appendix (Section \ref{ap}) we show the remaining assertion of Theorem \ref{RPF} by extending results from transport theory to the setting of linear spaces.

Our thanks to L. Cioletti  for many helpful conversations during the writing of our paper.

\section{Preliminaries}
\label{preliminaries}

Let $X$ be a Banach space and $T: X \to X$ be a linear continuous operator.  Hereafter, with exception of some specific cases, we will denote by $\| \cdot \|_X$ the norm for $X$.
We say that  the map $T$ is {\it positively expansive} in a subspace $Y$ of $X$, if there exists a constant $e > 1$ such that, for each point $\widetilde{x} \in S_Y = \{x \in Y ,\; \|x\|_X = 1\}$, there exists $n \in \mathbb{N}$ with $\|T^n(\widetilde{x})\|_XY\geq e$. Furthermore, $T$ it is said to be {\it uniformly positively expansive}  if there exist $e > 1$ and $m \in \mathbb{N}$ such that, $\|T^m(x)\|_Y \geq e$, for all $x$ in $S_Y$. Here we will assume this property when the subspace $Y$ is the  complement of the kernel of $T$.

Let $c_0(\mathbb{R})$  be the set of real sequences $(x_n)_{n \geq 1}$ such that $\displaystyle \lim_{n \to \infty}  x_n= 0$.  It is widely known that the vector space $c_0(\mathbb{R})$ equipped with the  supremum norm
\[
\|x\|_{c_0(\mathbb{R})} := \sup_{n \geq 1} \vert x_n \vert \;
\]
is a separable, non-reflexive Banach space.
On other hand, we also consider $l^p(\mathbb{R})$, for  $1 \leq p < \infty$, defined as the set of real sequences $(x_n)_{n \geq 1}$ satisfying $ \sum_{n = 1}^{\infty}|x_n|^p < \infty$, equipped with the norm
\[
\|x\|_{l^p(\mathbb{R})} := \left( \sum_{n = 1}^{\infty}|x_n|^p \right)^{\frac{1}{p}} \;.
\]

It is known that $l^p(\mathbb{R})$ is a separable Banach space and that it is reflexive when $1 < p < \infty$.   Our results will include the cases of the Banach spaces $l^1(\mathbb{R})$ and $c_0(\mathbb{R})$ which are separable but not reflexive Banach spaces. In the particular case where $p=2$, the space $l^2(\mathbb{R})$ is a separable Hilbert space when it is equipped with the following inner product
\[
\left\langle x, y \right\rangle_{l^2(\mathbb{R})}\; := \sum_{n =1}^{\infty}x_n y_n \;.
\]

Let $\{e_k\}_{k \geq 1}$ be such that each $e_k$ is a vector of the form $e_k = (\delta_{ik})_{i \geq 1}$. Then, any $x \in X \in \{c_0(\mathbb{R}),\; l^p(\mathbb{R}),\; 1 \leq p < \infty$\} can be written as
$$
\displaystyle x = \sum_{k =1}^{\infty} x_k e_k \;,
$$
where the series above converges in the norm $\|\cdot\|_X$, i.e. $\{e_k\}_{k \geq 1}$ is a Schauder basis for $X$ (in particular, $x_k = \left\langle x, e_k\right\rangle$ for any $k \in \mathbb{N}$ when $p = 2$).

Fix values $0 < c < c'$ and consider a sequence $(\alpha_n)_{n \geq 1}$ satisfying $\alpha_n \in (c, c')$ for each $n \in \mathbb{N}$. The {\it weighted shift} associated to the sequence $(\alpha_n)_{n \geq 1}$  is defined as the linear map $L : X \to X$, where $X \in \{c_0(\mathbb{R}),\;  l^p(\mathbb{R}),\; 1 \leq p < \infty\}$, and
$$
L((x_n)_{n \geq 1}) = (\alpha_n x_{n+1})_{n \geq 1} \;.
$$
%
Note that
$
L(e_1)=0$   and $L(e_n)= \alpha_{n-1} e_{n-1}$,  for all $n \geq 2$.
Besides that, for each $x \in X$ we have $\|L(x)\|_X \leq c'\|x\|_X$, thus, $L$ is a linear continuous operator and $\mathrm{Ker}(L) = \mathrm{span}\{e_1\}$. Moreover, in the case $c > 1$, the linear operator $L$ is {\it positively expanding} on  $E= \mathrm{span} \{e_n,\; n \geq 2\}$  with expanding constant equal to $c$, that is, $\|L(x)\|_X \geq c\|x\|_X$ for all $x \in E$.

  We remark that $L$ as above has always a one dimensional kernel but would like to point out that the results we get in the next two sections can be easily extended to the case where the kernel is finite dimensional. In Section \ref{extension} we will consider an even  more general case of linear operators with non-trivial kernels defined on a separable Banach space.

 From now on, we will use the notation
\[
\beta_k^n := \alpha_k ... \alpha_{k+n-1} \mbox {, for all } k, n  \in \mathbb{N}
\]
 and

\begin{eqnarray}
\label{dnn}
\displaystyle d_n := \inf_{k \geq 1} \beta_k^n \mbox {, for all } n \in \mathbb{N} \;.
\end{eqnarray}


It is widely known (see for instance \cite{Bo}), that the spectrum of $L : X \to X$, when $X$ is $c_0(\mathbb{R})$ or $l^p(\mathbb{R}),\; 1 \leq p \leq \infty$, is $\sigma(L)= D(0, r(L))$, where
$D(0, r(L))$ is the closed disc of center $0$ and radius $r(L)$ and $r(L)$ is the spectral radius of $L$ given by
\begin{equation}
r(L) := \lim_ {n \to \infty} \|L^n\|_{op}^{\frac{1}{n}} = \lim_{n \to \infty } \Bigl( \sup_{k \geq 1} \beta_k^n \Bigr)^{\frac{1}{n}} \;,
\label{spectral-ratio}
\end{equation}
where the operator norm $\|\cdot\|_{op}$ is given by
$$
\|L^n\|_{op} := \sup \{\|L^n(x) \|_X,\; \|x\|_X = 1\} \;.
$$

The following characterization is classical and will be useful in the following sections in order to guarantee the existence of eigenfunctions and invariant probabilities associated to the Ruelle operator (and its corresponding dual) which is the main tool to be used in here.

\begin{remark}
\label{swei} The asymptotic behavior of $\beta_k^n$ implies the following for a 
  weighted shift $L$ defined on $c_0(\mathbb{R})$ or $l^p(\mathbb{R})$, for $1 \leq p < \infty$. 
\noindent
\begin{enumerate}
\item
$L$ is topologically transitive if and only if $ \limsup_{n \to \infty} \beta_1^n = \infty$ (\cite{BaMa},  \cite{GrMa}).
\item
$L$ is topologically  mixing if and only if $\ \lim_ {n \to \infty} \beta_1^n = \infty$ (\cite{BeMe}).
\item
$L$ is frequently hypercyclic  in $l^{p}(\mathbb{R})$ if and only if $ \sum_{n=1}^{\infty} (\beta_1^n)^{-p}  < \infty $ (\cite{BR}).
\item $L$ is 
Devaney Chaotic   if and only if $ \sum_{n=1}^{\infty} (\beta_1^n)^{-p}  < \infty $ (\cite{GrMa}).
\item $L$ is 
positively expansive if and only if $ \sup_{n \geq 1} \beta_1^n = \infty$ (\cite{BCM}).
\end{enumerate}
\end{remark}


Given a separable Banach space $X$, we will use the notation $\mathcal{C}(X)$ for the set of continuous functions from $X$ into $\mathbb{R}$ and $\mathcal{C}_b(X)$ for the set of bounded continuous functions from $X$ into $\mathbb{R}$.  As it is well known, $\mathcal{C}_b(X)$ equipped with the uniform norm $\|\cdot\|_{\infty}$ given by $\|\varphi\|_{\infty} := \sup\{|\varphi(x)| : x \in X\}$ is a Banach space.

Given a metric $D$ on $X$, we will denote by $\mathrm{Lip}(D, X)$ the set of Lipschitz continuous functions from $X$ into $\mathbb{R}$. That is, the set of continuous functions satisfying
\[
\mathrm{Lip}_{\varphi, D} := \sup\left\{\frac{|\varphi(x) - \varphi(y)|}{D(x, y)} :\;  x \neq y \right\} < \infty \;.
\]

We will use the notation $\mathrm{Lip}_b(D, X)$ for the set of bounded Lipschitz continuous functions from $X$ into $\mathbb{R}$. Note that for any $0 < \alpha \leq 1$, the set $\mathrm{Lip}(D^{\alpha}, X)$ is the set of $\alpha$-H\"older continuous functions from $X$ into $\mathbb{R}$ and the set $\mathrm{Lip}_b(D^{\alpha}, X)$ is the set of bounded $\alpha$-H\"older continuous functions from $X$ into $\mathbb{R}$, where $D^{\alpha}(x, y) := (D(x, y))^{\alpha}$.

Besides that, given $\delta > 0$, we will use the notation $\mathrm{Lip}^{loc(\delta)}(D, X)$ for the set of continuous functions from $X$ into $\mathbb{R}$, such that
\[
\mathrm{Lip}^{loc(\delta)}_{\varphi, D} := \sup \left\{ 
\frac{|\varphi(x) - \varphi(y)|}{D(x, y)} :  x \neq y,\;D(x, y) < \delta \right\} < \infty.
\]
In addition, we will denote by $\mathrm{Lip}^{loc(\delta)}_b(D, X)$ the set of bounded functions in $\mathrm{Lip}^{loc(\delta)}(D, X)$ and, from now on, we will denote by $D_X$ the metric generated by the norm $\|\cdot\|_X$, that is,
$$
D_X(x, y) = \|x - y\|_X \;.
$$
Observe that the Riesz representation theorem in this setting implies that the dual of $\mathcal{C}_b(X)$, which we will denote by $\mathcal{C}_b(X)^*$, coincides with the set of additive finite Borel signed measures on $X$, denoted by $\mathcal{B}_{\mathcal{A}}(X)$ in here. In addition, we will use the notation $\mathcal{B}(X)$ for the set of sigma additive measures and $\mathcal{P}(X)$ for the set of Borel sigma additive probability measures on $X$.

\section{A RPF Theorem for weighted shifts}
\label{RPF-theorem}

The Ruelle operator is a powerful tool in thermodynamic formalism, especially for H\"older potentials and expanding dynamical systems (in the sense of Ruelle) defined on a compact metric space (see \cite{PP}). Results of this nature were also obtained for the case where the number of preimages of the underlying dynamics of each point is uncountable as, for example, in the case of  the generalized  $X Y$ model via the use of an {\it a priori} probability (see for instance \cite{BCLMS,LMMS,MU,ACR}). Even in the case of a non-compact alphabet $\mathcal{A}$, one can get in specific cases similar results for $\mathcal{A}^\mathbb{N}$ (see \cite{LV,FV,CSS}).

Adapting some ideas from the classical setting (as for instance in \cite{LV}) we show in this section the existence of invariant probabilities for the above defined linear dynamical system via the Ruelle operator. After that, in  section  \ref{extension}, we will show  that our approach can be adapted to the case where the system is given by  a class of frequently hypercyclic operators (see \cite{Gri} and \cite{CGP}) with a finite dimensional kernel.

We now return to the  weighted shift $L:   X \to X$  where $X$ is  $c_0(\mathbb{R})$ or $ l^p(\mathbb{R})$, for $1 \leq p < \infty$.
The analogue of the Ruelle operator on the set $\mathcal{C}_b(X)$ is defined in the following way. We fix an {\it a priori} probability measure $m$ on the Borel sets of $\mathbb{R}$ equipped with the usual topology and assume that the support of the probability measure $m$ is equal to the set $\mathbb{R}$.
%
Since $\mathrm{Ker}(L)$ is isometrically isomorphic with $\mathbb{R}$ it follows that $m$ induces an {\it a priori} probability measure on $\mathrm{Ker}(L)$ - which we will also denote by the same letter $m$.

As an example, one could take $m$ as the Gaussian distribution on $\mathbb{R}$ with mean zero and variance $1$. That is $m = f\;dr$, with $f(r) = \frac{1}{\sqrt{2\pi}}e^{-r^2/2}$.

We say that a potential $A \in \mathcal{C}(X)$ has {\it summable variation} if 
\begin{equation}
\label{summable-variation-weighted}
V(A) = \sum_{n = 1}^{\infty} V_n(A) < \infty \;,
\end{equation}
where $V_n(A) = \sup\{|A(x) - A(y)| : x, y \in X,\; x_i = y_i,\; 1 \leq i \leq n\}$. Note that \eqref{summable-variation-weighted} does not imply that $A$ is a bounded potential. We will denote by $\mathcal{SV}(X)$ the set of potentials $A \in \mathcal{C}(X)$ satisfying \eqref{summable-variation-weighted}.

\begin{remark} We would like to point out that in the setting of linear dynamics in contrast to shift spaces (even with respect to uncountable, non-compact alphabets as in \cite{CSS}), there are bounded Lipschitz functions which are not in $\mathcal{SV}(X)$. For example, the function $A(x) := \arctan \|x\|$ is an element of $\mathrm{Lip}_b(D_X, X)$, but as $V_n(A) = \pi$ for any $n \in \mathbb{N}$, the function $A$ does not have summable variations.
On the other hand, $\mathrm{Lip}_b(D_X^\alpha, X) \cap \mathcal{SV}(X)$ is non-empty which follows from the following construction. For $x,y \in X$, set 
\[ 
D_{\hbox{\tiny shift}}(x,y) := 
\begin{cases}
\left( \sum_n \min(\{1,|x_n - y_n|^p  \}) 2^{-n} \right)^{1/p} &: X= l^p(\mathbb{R}) \\
\sup (\{ \min(\{1,|x_n - y_n| \}) 2^{-n} : n \in \mathbb{N} \})  &: X= c_0(\mathbb{R}) 
\end{cases}
\]
Note that $D_{\hbox{\tiny shift}}(x,y) \leq D_X (x,y)$ for any pair $x,y$ with $\sup |x_n - y_n| \leq 1$ and that, as $D_{\hbox{\tiny shift}}(x,y) \leq 1$ 
for any $x,y \in X$, each function in  $ \mathrm{Lip}(D_{\hbox{\tiny shift}}^\alpha, X)$ is bounded. Moreover, it follows from  $\sum_n 2^{-n} < \infty$  that any function in  $ \mathrm{Lip}(D_{\hbox{\tiny shift}}^\alpha, X)$ has   
summable variations. Hence, 
\[ \mathrm{Lip}(D_{\hbox{\tiny shift}}^\alpha, X) \subset \mathrm{Lip}_b(D_X^{\alpha}, X) \cap  \mathcal{SV}(X).\] 
Finally, as any function in $\mathrm{Lip}_b(D_X^{\alpha}, X)$ which only depends on a finite number of coordinates of $X$ 
is an element of $\mathrm{Lip}(D_{\hbox{\tiny shift}}^\alpha, X)$, the set of $D_X^\alpha$-Lipschitz function with summable variations is non-trivial.
These observations also provide an explanation of the phenomena that shift spaces whose alphabet is a standard Borel space show an exponential contraction under the iteration of the transfer operator  (see \cite{CSS}). However, this is not the case in the setting of linear dynamics (see Remark \ref{esta1}). Namely, we obtain contraction to zero with time, but without an explicit decay rate,  as a consequence of Theorem \ref{theo:conformal-measure}. 
\end{remark}

We now define the {\it Ruelle operator}. 

\begin{definition}
Given the {\it a priori} probability $m$ and a potential $A \in \mathrm{Lip}_b(D_X^{\alpha}, X)$, the Ruelle operator is defined as the map assigning to each function  $\varphi \in \mathcal{C}_b(X)$ the function
\begin{equation}
\mathcal{L}_A(\varphi)(x) := \int_{Lv = x} e^{A(v)}\varphi(v)\; d m(v) \;.
\label{Ruelle-operator}
\end{equation}
\end{definition}

Using the coordinates induced by the Schauder basis $\{e_k\}_{k \geq 1}$ it is possible to write the equation \eqref{Ruelle-operator} \ms{as}
\begin{equation}
\mathcal{L}_A(\varphi)(x_1, x_2, ...) := \int_{\mathbb{R}} e^{A\bigl(r, \frac{x_1}{\alpha_1}, \frac{x_2}{\alpha_2}, ...\bigr)}\varphi\bigl(r, \frac{x_1}{\alpha_1}, \frac{x_2}{\alpha_2}, ...\bigr) d m(r) \nonumber \;.
\end{equation}
It is easy to  verify that $\mathcal{L}_A(\varphi) \in \mathcal{C}_b(X)$ (see the proof of Lemma \ref{lemma-Holder}) and that the $n$-th iterate ($n \geq 2$) of the Ruelle operator is given by 
\begin{align*}
\mathcal{L}^n_A(\varphi)(x) 
= & \int_{\mathbb{R}^n} e^{\sum_{j = 1}^n A\Bigl(r_j, \frac{r_{j-1}}{\beta^1_1}, ... , \frac{r_1}{\beta^{j-1}_1}, \frac{x_1}{\beta^j_1}, \frac{x_2}{\beta^j_2}, ...\Bigr)} \\
& \quad \times \varphi\Bigl(r_n, \frac{r_{n-1}}{\beta^1_1}, ... , \frac{r_1}{\beta^{n-1}_1}, \frac{x_1}{\beta^n_1}, \frac{x_2}{\beta^n_2}, ...\Bigr) dm(r_1) \cdots dm(r_n).
\end{align*}

\begin{definition} We say that a potential $A: X \to \mathbb{R}$ is normalized if  $\mathcal{L}_A(1)=1$.
\end{definition}

\begin{remark}\label{remark:normalization}

Basic examples for normalized potentials are  $A=0$ or any potential which only depends on the first coordinate, that is 
$A(r_1,r_2,r_3,..)=A(r_1)$ with respect to the coordinates induced by the Schauder basis, and such that $\int e^{A(r_1) }d m (r_1)=1$. A further important example is given by the following normalization process. Assume that 
$\mathcal{L}_A(f) = \lambda f $ for some $\lambda > 0$ and $f: X \to (0,\infty)$. Then, for $\overline{A} := A + \log f - \log f \circ L - \log \lambda$ and $x\in X$,
\[
\mathcal{L}_{\overline{A}}(1)(x) = \frac{\mathcal{L}_A(f)(x)}{\lambda f(x)} = \frac{\lambda f(x)}{\lambda f(x)} =1.
\]
Hence, $\overline{A}$ is normalized and $\overline{A}$ is therefore known as the normalization of $A$.  
Finally, note that, if $A$ is normalized, then the dynamics of $L$ and the action of $\mathcal{L}_A$ on $ \mathcal{C}_b(X)$ are connected through 
\[
\mathcal{L}_A^n(\varphi \circ L^n) = \varphi, \quad \forall n \in \mathbb{N}.
\]
\end{remark}

In the next Lemma we will show that for any  potential $A \in \mathrm{Lip}_b(D_X^{\alpha}, X)$, the Ruelle operator $\mathcal{L}_{A}$ preserves the set of $\alpha$-H\"older continuous functions.

\begin{lemma}
\label{lemma-Holder}
Consider $X =c_0(\mathbb{R})$, or $X= l^p(\mathbb{R}), 1 \leq p < \infty,$ and $A \in \mathrm{Lip}_b(D_X^{\alpha}, X)$. Then, for each $n \geq 1$ and any $\delta > 0$, the $n$-th iterate of the Ruelle operator preserves 
$\mathrm{Lip}_b(D_X^{\alpha}, X)$ and $\mathrm{Lip}^{loc(\delta)}_b(D_X^{\alpha}, X)$. Moreover, for
\[ D_n := \sum_{j=1}^n (d_j)^{-\alpha}, \ \ C_A:= \mathrm{Lip}_{A, D_X^{\alpha}}  \]
we obtain the following estimates.
\begin{enumerate}[a)]
\item For any $\varphi \in \mathrm{Lip}_b(D_X^{\alpha}, X)$ and $x,y \in X$, $x \neq y$ we have 
{
\begin{equation}
\label{n-th-Holder-constant}
\frac{|\mathcal{L}^n_A(\varphi)(x) - \mathcal{L}^n_A(\varphi)(y)|}{\|x-y\|_X \mathcal{L}^n_A(1)(x) 
} 
\leq
(d_n)^{-\alpha} \mathrm{Lip}_{\varphi, D_X^{\alpha}} + \|\varphi \|_\infty
C_A D_n \frac{e^{2n \|A\|_\infty} -1}{2n \|A\|_\infty}. 
\end{equation} 
}
\item For any $\varphi \in \mathrm{Lip}^{loc(\delta)}_b(D_X^{\alpha}, X)$ and $x,y \in X$, $x \neq y$ and $\|x-y\|_X< \delta $, we have  
\begin{equation}
\label{n-th-local-Holder-constant}
\frac{|\mathcal{L}^n_A(\varphi)(x) - \mathcal{L}^n_A(\varphi)(y)|}{\|x-y\|_X \mathcal{L}^n_A(1)(x) 
} 
\leq
(d_n)^{-\alpha} \mathrm{Lip}_{\varphi, D_X^{\alpha}} + \|\varphi \|_\infty
\frac{e^{C_A D_n \delta^\alpha } -1}{\delta^\alpha}
.
\end{equation} 
\end{enumerate}
\end{lemma}

\begin{proof}
 First note that for $x= (x_i)$, $y= (y_i) \in X$,
\begin{align*}
& D_X
\left(\left(r_j, \ldots  {r_1}/{\beta^{j-1}_1} , {x_1}/{\beta^j_1}, \ldots \right), \left(r_j, \ldots  {r_1}/{\beta^{j-1}_1} , {x_1}/{\beta^j_1}, \ldots \right) \right)
\\
= & \left\|
\left(  {x_1}/{\beta^j_1}, {x_2}/{\beta^j_2} \ldots \right) - \left(  {y_1}/{\beta^j_1}, {y_2}/{\beta^j_2} \ldots \right)\right\|_X \leq (d_j)^{-1} D_X(x,y). 
\end{align*}
It now follows from Hölder continuity of $A$ that 
\begin{align*}
(\ast) & := \left|
e^{\sum_{j = 1}^n A\left(r_j, \ldots  {r_1}/{\beta^{j-1}_1} , {x_1}/{\beta^j_1}, \ldots \right)} 
- 
e^{\sum_{j = 1}^n A\left(r_j, \ldots  {r_1}/{\beta^{j-1}_1} ,{y_1}/{\beta^j_1}, \ldots \right)
} 
\right|
\\
& = 
e^{\sum_{j = 1}^n A\left(r_j, \ldots  {x_1}/{\beta^j_1}, \ldots \right)} 
\left| 
e^{\sum_{j = 1}^n A\left(r_j, \ldots  {x_1}/{\beta^j_1}, \ldots \right) - A\left(r_j, \ldots  {y_1}/{\beta^j_1}, \ldots \right) }
- 1
\right|
\\
& \leq 
e^{\sum_{j = 1}^n A\left(r_j, \ldots  {x_1}/{\beta^j_1}, \ldots \right)} 
\left| 
e^{   
\mathrm{Lip}_{\varphi, D_X^{\alpha}}
\sum_{j = 1}^n (d_j)^{-\alpha} 
 \|x - y\|_X^\alpha
}
- 1
\right|.
\end{align*}
The above then implies that 
\[
(\ast) \leq 
\frac{e^{C_A D_n  \|x - y\|_X^\alpha } -1}{\|x - y\|_X^\alpha}
e^{\sum_{j = 1}^n A\left(r_j, \ldots  {x_1}/{\beta^j_1}, \ldots \right)} \|x - y\|_X^\alpha,
\]
which provides an effective estimate as long as $\|x - y\|_X^\alpha$ is uniformly bounded. That is, for 
$\|x - y\|_X \leq \delta$, 
it follows that
\begin{align*}
 &\bigl| \mathcal{L}_A^n(\varphi)(x) - \mathcal{L}_A^n(\varphi)(y)\bigr|
   \\
&
= \Bigl| \int_{\mathbb{R}^n} e^{\sum_{j = 1}^n A\Bigl(r_j, ... , \frac{r_1}{\beta^{j-1}_1}, \frac{x_1}{\beta^j_1}, ...\Bigr)}\varphi\Bigl(r_n, ... , \frac{r_1}{\beta^{n-1}_1}, \frac{x_1}{\beta^n_1}, ...\Bigr) 
\\
& \ \ -  e^{\sum_{j = 1}^n A\Bigl(r_j, ... , \frac{r_1}{\beta^{j-1}_1}, \frac{y_1}{\beta^j_1}, ...\Bigr)}\varphi\Bigl(r_n, ... , \frac{r_1}{\beta^{n-1}_1}, \frac{y_1}{\beta^n_1}, ...\Bigr) dm(r_1) \cdots dm(r_n) \Bigr| 
\\
&\leq \int_{\mathbb{R}^n} e^{\sum_{j = 1}^n A\Bigl(r_j, ... \frac{x_1}{\beta^j_1} ...\Bigr)}
\bigl|\varphi\Bigl(r_n, ...  \frac{x_1}{\beta^n_1}, ...\Bigr) - \varphi\Bigl(r_n, ...   \frac{y_1}{\beta^n_1}, ...\Bigr)\bigr| dm(r_1) \cdots dm(r_n) 
\\
 &\phantom{\leq} + \|\varphi\|_\infty \int_{\mathbb{R}^n} (\ast) \,
 dm(r_1) \cdots dm(r_n) \\
  &\leq \mathcal{L}^n_A(1)(x) \|x - y\|_X^\alpha \left( d_n^{-\alpha} \mathrm{Lip}_{\varphi, D_X^{\alpha}} 
    +   \|\varphi\|_\infty  \frac{e^{C_A D_n \delta^\alpha } -1}{\delta^\alpha}
    \right), 
\end{align*}
which proves the estimate in \eqref{n-th-local-Holder-constant}. 
 On the other hand, if $\|x - y\|_X^\alpha$ is large, then the uniform boundedness of $A$ implies that
 \[
 (\ast) \leq 
 \frac{e^{2n \|A\|_\infty} -1}{\|x - y\|_X^\alpha}
 e^{\sum_{j = 1}^n A\left(r_j, \ldots  {x_1}/{\beta^j_1}, \ldots \right)} \|x - y\|_X^\alpha
 ,
 \]
In  particular, by choosing $\delta$ such that  
\[
 \frac{e^{2n \|A\|_\infty} -1}{\delta^\alpha} =  \frac{e^{C_A D_n \delta^\alpha } -1}{\delta^\alpha},
\] 
that is $\delta^\alpha = 2n \|A\|_\infty/(C_A D_n)$
one obtains that 
\[
 \min \left\{ \frac{e^{2n \|A\|_\infty} -1}{\delta^\alpha}, 
 \frac{e^{C_A D_n \delta^\alpha }  -1}{\delta^\alpha}
 \right\} \leq C_A D_n  \frac{e^{2n \|A\|_\infty} -1}{2n \|A\|_\infty}. 
 \]
The estimate \eqref{n-th-Holder-constant} follows from this. With respect to the invariance of $\mathrm{Lip}_b(D_X^{\alpha}, X)$ and $\mathrm{Lip}^{loc(\delta)}_b(D_X^{\alpha}, X)$, it suffices to observe that $\|\mathcal{L}^n_A(1)\|_\infty \leq \exp(n \|A\|_ \infty)< \infty$ and apply .
\eqref{n-th-Holder-constant} and \eqref{n-th-local-Holder-constant}, respectively.
\end{proof}

\begin{remark}
\label{pq}
Consider $1 \leq p < q < \infty$, since $l^p(\mathbb{R}) \subset l^q(\mathbb{R})$ and $\|x \|_{l^q(\mathbb{R})} \leq \|x \|_{l^p(\mathbb{R})}$ for all $x \in l^p(\mathbb{R})$, we deduce that if $A \in \mathrm{Lip}_b(D_X^{\alpha}, l^q(\mathbb{R}))$, then  the Ruelle operator preserves the set $\mathrm{Lip}_b(D_X^{\alpha}, l^p(\mathbb{R}))$.
\end{remark}


\begin{remark} We would like to remark now that the reason to consider the action of $\mathcal{L}_A$ on functions which are locally Lipschitz continuous is based on the fact that 
the second term in \eqref{n-th-local-Holder-constant} is uniformly bounded whenever $\sup_n D_n < \infty$. Moreover, a further advantage will become apparent in the appendix (Section \ref{ap}), where we will consider a bounded metric $\widetilde{D}$ on $X$ which is equivalent to the metric {$D_X(x, y)=\|x - y\|_X$, where $\|\cdot\|_X$} stands for norm on $X$. The necessity for this change of metric is twofold. The first reason is based on the simple observation that the  Wasserstein space with respect to a bounded metric contains all Borel probability measures (see, e.g., \cite{Villani09}) and therefore, provided that $A$ is normalized, the dual $\mathcal{L}_A^*$  acts on this space. The second reason is of technical nature as this allows to obtain contraction for probability measures whose supports are arbitrarily distant.    

However, the new metric also requires a change in the underlying function spaces. 
Namely, as it will turn out below,  the space of bounded, locally Lipschitz functions with respect to $D$ will coincide with the space of  Lipschitz functions with respect to $\widetilde{D}$.
\end{remark}
\smallskip

We now define the dual $\mathcal{L}_{A}^\ast$ of $\mathcal{L}_A$ based on the observation that  $\mathcal{L}_A$ acts on the space of bounded Lipschitz functions by Lemma \ref{lemma-Holder}. That is,  
\[ \varphi \to \int \varphi d\mathcal{L}_{A}^\ast(\mu) := \int \mathcal{L}_{A}(\varphi) d \mu \]
defines an action of on the space of finite measures.  
Furthermore, if $A$ is normalized, then $\mathcal{L}_{A}^\ast$ acts on the space of sigma-additive probability measures. In here, we put emphasis on sigma-additivity in order to avoid any ambiguity with finitely additive measures, which are canonical in the similar setting of \cite{CSS}.

\begin{definition} \label{Ro} Given a normalized potential $A: X \to \mathbb{R}$, the sigma-additive probability $\mu_A$ on $X$, such that, $\mathcal{L}_{A}^*(\mu_A)=\mu_A$, will be called    the Gibbs probability for the potential $A$.
\end{definition}

We emphasize here the fact that the {\it Gibbs probability} terminology can have several different meanings depending on the authors of each work under consideration. Here the concept appears when this invariant probability is obtained from Ruelle's Theorem.
The terminology {\it equilibrium probability}  (not considered here) would be reserved for those derived from a principle of maximizing pressure.

Observe  that Remark \ref{remark:normalization} immediately implies that a Gibbs probability  $\mu_A$ is always $L$-invariant. Furthermore, the existence of eigenfunctions of $\mathcal{L}_{A}$ and eigenprobabilities of $\mathcal{L}_{A}^\ast$, as given in Theorem \ref{RPF} below,
provides important tools for solving problems related to the construction of equilibrium probabilities, that is invariant probabilities that maximize the pressure (see \cite{LV1}).




Before the statement of our main result
we introduce the class of weighted shifts $L$ for which Theorem \ref{RPF} holds.
Fix  $0 < c < c'$, with $1 < c'$ and consider a sequence $(\alpha_n)_{n \geq 1}$ satisfying $\alpha_n \in (c, c')$, for each $n \in \mathbb{N}$.
Suppose also that
\begin{eqnarray}
\label{d}
\displaystyle \sum_{n =1}^{\infty} (d_n)^{-\alpha} = d < \infty \;,
\end{eqnarray}
where $d_n$ is given in \eqref{dnn}.
Observe that if $c > 1$, then \eqref{d} is always satisfied. Moreover, since \eqref{d} implies that
$$
\sum_{n =1}^{\infty} (\beta_1^n)^{-p} < \infty\;,
$$
 we deduce as a consequence of  Remark \ref{swei} that the weighted shift $L$  associated to the sequence $(\alpha_n)_{n \geq 1}$ is
uniformly positively expansive and topologically mixing, Devaney chaotic and frequently hypercyclic  (see for instance \cite{BeMe}, \cite{G}, and \cite{GrMa}).
In order to guarantee the existence of invariant measures, we need the following property.
\begin{definition}  \label{Def:adapted}
We say that $m$ has {\it adapted tails} if, for any $\epsilon > 0$ there exists a sequence of positive numbers $(\kappa_n)_{n \geq 1}$, such that, 
\begin{enumerate}
\item $\sum_{n=1}^\infty m(\mathbb{R} \setminus [-\beta_1^n \kappa_n,\beta_1^n \kappa_n]) < \epsilon$;
\item the sequence $ (\kappa_n)_{n \geq 1} $ is in $X$. 
\end{enumerate}
\end{definition}
For examples of measures with adapted tails we refer to Propositions \ref{prop:polynomial-tails} and \ref{prop:exponential-tails}. Recall that the support of a Borelian measure $\nu$ on a metric space $(D, X)$ is defined by $
\hbox{supp}(\nu) := \left\{x \in X :  \nu(\{y \in X: D(x, y)< \epsilon \}) > 0, \;\forall \epsilon > 0  \right\}$.

\begin{theorem} 
\label{RPF}
Assume that $X = c_0(\mathbb{R})$ or $X= l^p(\mathbb{R})$ with  $1 \leq p < \infty$,  that  $L : X \to X$ is a weighted shift satisfying \eqref{d} and that  $A \in \mathrm{Lip}_b(D_X^{\alpha}, X)$.
\begin{enumerate}[a)]
\item If $A \in \mathcal{SV}(X) $, then there exist $\lambda_A > 0$ and  $\psi_A \in \mathrm{Lip}_b(D_X^{\alpha}, X)$ with $\inf_{x \in X} \psi_A (x) > 0$, such that, $\mathcal{L}_A (\psi_A)(x) = \lambda_A \psi_A(x)$, for each $x \in X$. Moreover, $\lambda_A$ is equal to the spectral radius of the action of $\mathcal{L}_A$ on $\mathrm{Lip}_b(D_X^{\alpha}, X)$, and any positive eigenfunction in 
$\mathrm{Lip}_b(D_X^{\alpha}, X)$ is colinear to  $\psi_A$.
\item If $m$ has adapted tails and $\mathcal{L}_A(1) = 1$, then there exists a unique $\sigma$-additive Borel probability measure $\mu_A$ which is a fixed point for the operator $\mathcal{L}^*_{A}$. Moreover, $\lim_{n \to \infty} (\mathcal{L}^n_{A})^\ast(\nu) = \mu_A$ in the weak-$\ast$ topology for any Borel probability measure $\nu$. In particular, $\mu_A$ is an invariant and mixing probability measure  
 (and therefore conservative and ergodic).
\item If $A \in \mathcal{SV}(X) $ and  $m$ has adapted tails, then  there exists a unique fixed point $\mu_A$ for the operator $\mathcal{L}^*_{\overline{A}}$, where   $\overline{A} = A + \log(\psi_A) - \log(\psi_A \circ L) - \log(\lambda_A)$.  
Moreover,  $\mu_A$ is an invariant and mixing sigma-additive probability measure. 
\item If $A \in \mathcal{SV}(X)$ and $\pi_1$ is  the projection on the first coordinate, we get
 $
 \hbox{supp}(\mu_A)  \supset   \pi_1^{-1}(\hbox{supp}(m)) 
 $. From this follows that $\hbox{supp}(\mu_A) = X$.
\end{enumerate}
\end{theorem} 

Before giving the proof, we remark that item c) of Theorem \ref{RPF} in combination with the normalization of Remark \ref{remark:normalization} allows us to extend Definition \ref{Ro} to potentials which are not necessarily normalized. That is, for $A \in  \mathcal{SV}(X)$ and $m$ with adapted tails, we refer to $\mu_A$ as in c) as the Gibbs measure for the potential $A$. 
Moreover, we would like to remark that item b) only requires that $A$ is a Hölder contiunous and normalized potential which is not necessatily of bounded distortion (see \eqref{summable-variation-weighted}).

\begin{proof}  We postpone the discussion of the existence of $\mu_A$ in item b) to Theorem  \ref{theo:conformal-measure} in  the Appendix (Section \ref{ap}), but the remaining statements of item b) are proven below.
For the remaining parts, we adapt   the results that appear in \cite{BCLMS} and \cite{LMMS}.

\smallskip
\noindent \textit{Part a)} 
For each $s \in (0, 1)$ define the operator $\mathcal{T}_{s, A} : \mathcal{C}_b(X) \to \mathcal{C}_b(X)$ by
\begin{equation}
\mathcal{T}_{s, A}(u)(x) = \log\left( \int_{Lv = x} e^{A(v) + su(v)} dm(v) \right) \;.
\label{discount}
\end{equation}

The use of this kind of operator is quite common in the so called discounted method (see Section $1$ in \cite{BCLMS}). In order to show existence of the eigenfunction  $\psi_A$ for the Ruelle operator $\mathcal{L}_A$  it will be required later to take $s \to 1$.

For fixed $s$, and using the  coordinates notation obtained from the Schauder basis $\{e_k\}_{k \geq 1}$, it follows that the equation \eqref{discount} can be written as
\begin{equation}
\mathcal{T}_{s, A}(u)(x_1, x_2, ...) = \log\left( \int_{\mathbb{R}} e^{A\bigl(r, \frac{x_1}{\alpha_1}, \frac{x_2}{\alpha_2}, ...\bigr) + su\bigl(r, \frac{x_1}{\alpha_1}, \frac{x_2}{\alpha_2}, ...\bigr)} d  m(r) \right) \;.
\label{discount2}
\end{equation}

Since the maps $A$ and $u$ are bounded continuous functions and $m$ is a probability measure, we deduce that $\mathcal{T}_{s, A}(u) \in \mathcal{C}_b(X)$ for all $u \in \mathcal{C}_b(X)$.

Now, observe that for all $x, y \in X$,
\begin{align}
|\mathcal{T}_{s, A}(u_1)(x) - \mathcal{T}_{s, A}(u_2)(x)|
&= \Bigl| \log\left( \frac{\int_{Lv = x} e^{A(v) + su_1(v)} dm(v)}{\int_{Lv = x} e^{A(v) + su_2(v)} dm(v)}\right)\Bigr| \nonumber \\
&\leq \log\left( \frac{\int_{Lv = x} e^{A(v) + su_2(v)+ s\|u_1 - u_2\|_{\infty}} dm(v)}{\int_{Lv = x} e^{A(v) + su_2(v)} d m(v)}\right) \nonumber \\
&= s\|u_1 - u_2\|_{\infty} \nonumber \;.
\end{align}

 This shows that for all  $0<s<1$, the operator  $\mathcal{T}_{s, A}$ is a uniform contraction. Therefore, since $\mathcal{C}_b(X)$ is a complete metric space with the distance induced by the uniform norm $\|\cdot\|_{\infty}$,  it follows from the Banach contraction principle that $\mathcal{T}_{s, A}$ has a unique fixed point $u_s \in \mathcal{C}_b(X)$. That is, for all $x$ and $0<s<1$,
\[
e^{u_s(x)} = \int_{Lv = x} e^{A(v) + su_s(v)} dm(v) \;.
\]

Now, we will show that the collection $(u^*_s)_{0 < s < 1}$, with
$$u^*_s(x) = u_s(x) - u_s(0),$$
for each $x \in X$, is equicontinuous and uniformly bounded.

By the Arzel\'a-Ascoli  Theorem, as a consequence of the Cantor-Tychonoff Theorem, we claim the existence of a sequence $(s_n)_{n \geq 1}$ satisfying
\begin{enumerate}[i)]
\item $\displaystyle \lim_{n \to \infty}s_n = 1$;
\item $\displaystyle \lim_{n \to \infty}u^*_{s_n}(x) = u(x)$ for each $x \in X$;
\item $\displaystyle \lim_{n \to \infty}u^*_{s_n} = u$ as uniform limit on any compact subset of $X$.
\end{enumerate}
We will now show that the function   $e^{u} = \psi_A$ is the main eigenfunction for $\mathcal{L}_A$. That is, $\psi_A$ satisfies $a)$, with $\lambda_A$ maximal on the set of eigenvalues of $\mathcal{L}_A$.

 In order to prove that  $(u^*_s)_{0 < s < 1}$ is equicontinuous,  for any pair $x, y \in X$, set
 \begin{align*}
 \mathcal{S}_1(x, y)
& := \sup_{r \in \mathbb{R}}\Bigl\{ A \bigl(r, \frac{x_1}{\alpha_1}, \frac{x_2}{\alpha_2},  ...\bigr) - A\bigl(r, \frac{y_1}{\alpha_1}, \frac{y_2}{\alpha_2}, ...\bigr) \\
& \phantom{:= \sup_{r \in \mathbb{R}}\Bigl\{} + s\bigl(u_s\bigl(r, \frac{x_1}{\alpha_1}, \frac{x_2}{\alpha_2}, ... \bigr) - u_s\bigl(r, \frac{y_1}{\alpha_1}, \frac{y_2}{\alpha_2},... \bigr)\;\bigr) \Bigr\} 
 \end{align*}
and note that $\sup_{x,y} \mathcal{S}_1(x, y) < \infty$ as $A$ and $u_s$ are bounded.
%
Using the expression in coordinates for the operator $\mathcal{T}_{s, A}$ (which appears in \eqref{discount2}) and the fact that $u_s$ is a fixed point for $\mathcal{T}_{s, A}$, it follows that
\begin{align}
e^{u_s(x)}
&= \int_{\mathbb{R}} e^{A\bigl(r, \frac{x_1}{\alpha_1}, \frac{x_2}{\alpha_2}, ...\bigr) + su_s\bigl(r, \frac{x_1}{\alpha_1}, \frac{x_2}{\alpha_2}, ...\bigr)} dm(r) \nonumber \\
&\leq  e^{\mathcal{S}_1(x, y)} \int_{\mathbb{R}} e^{A\bigl(r, \frac{y_1}{\alpha_1}, \frac{y_2}{\alpha_2}, ...\bigr) + su_s\bigl(r, \frac{y_1}{\alpha_1}, \frac{y_2}{\alpha_2}, ...\bigr)} dm(r)  \nonumber \\
&= e^{\mathcal{S}_1(x, y)} e^{u_s(y)}\nonumber \;.
\end{align}
That is, 
$$|u_s(x) - u_s(y)| \leq \mathcal{S}_1(x, y) < +\infty\;.$$
Therefore, following an inductive argument on the first $n$ coordinates of the points in which it is calculated the function $A + su_s$, it is easy to check that for any $n \in \mathbb{N}$,   the inequality 
\begin{equation}
\label{n-th-variation-sum}
|u_s(x) - u_s(y)| \leq \mathcal{S}_n(x, y) < +\infty\;,
\end{equation} 
is satisfied, where
\begin{align}
 \mathcal{S}_n(x, y) & : =  
\sup_{(r_1, ..., r_n)} \Bigl\{\Bigl(\sum_{j = 1}^n s^{j-1} \Bigl(A\bigl(r_j, \frac{r_{j-1}}{\beta^1_1}, ... , \frac{r_1}{\beta^{j-1}_1}, \frac{x_1}{\beta^j_1}, \frac{x_2}{\beta^j_2}, ...\bigr) \nonumber \\
& \phantom{: =  \sup_{(r_1, ..., r_n)} \Bigl\{\Bigl(\sum_{j = 1}^n s^{j-1} } 
- A\bigl(r_j, \frac{r_{j-1}}{\beta^1_1}, ... , \frac{r_1}{\beta^{j-1}_1}, \frac{y_1}{\beta^j_1}, \frac{y_2}{\beta^j_2}, ...\bigr)\Bigr)\Bigr) \nonumber \\
& \phantom{: =  \sup_{(r_1, ..., r_n)} \Bigl\{ }
+ s^n\Bigl(u_s \bigl(r_n, \frac{r_{n-1}}{\beta^1_1}, ... , \frac{r_1}{\beta^{n-1}_1}, \frac{x_1}{\beta^n_1}, \frac{x_2}{\beta^n_2}, ...\bigr) \nonumber \\
& \phantom{: =  \sup_{(r_1, ..., r_n)} \Bigl\{  + s^n\Bigl( }
- u_s\bigl(r_n, \frac{r_{n-1}}{\beta^1_1}, ... , \frac{r_1}{\beta^{n-1}_1}, \frac{y_1}{\beta^n_1}, \frac{y_2}{\beta^n_2}, ...\bigr)\Bigr) \Bigr\} \nonumber \;.
\end{align}
Moreover, for any pair $x, y \in X$, the sequence $(\mathcal{S}_n(x, y))_{n \geq 1}$ satisfies
\begin{align}
&\mathcal{S}_n(x, y) \nonumber \\
&\leq \sup_{(r_1, ..., r_n)} \Bigl\{\Bigl(\sum_{j = 1}^n s^{j-1}\mathrm{Lip}_{A, D_X^{\alpha}} \Bigl\|\Bigl(r_j, \frac{r_{j-1}}{\beta^1_1}, ... , \frac{r_1}{\beta^{j-1}_1}, \frac{x_1}{\beta^j_1}, \frac{x_2}{\beta^j_2}, ...\Bigr) \nonumber \\
&\phantom{\leq} \; - \Bigl(r_j, \frac{r_{j-1}}{\beta^1_1}, ... , \frac{r_1}{\beta^{j-1}_1}, \frac{y_1}{\beta^j_1}, \frac{y_2}{\beta^j_2}, ...\Bigr) \Bigr\|_X^{\alpha}\Bigr)  + 2s^n \|u_s\|_{\infty} \Bigr\} \nonumber \\
&= \Bigl(\sum_{j = 1}^n s^{j-1}\mathrm{Lip}_{A, D_X^{\alpha}} \Bigl\|\Bigl(\frac{x_1}{\beta^j_1}, \frac{x_2}{\beta^j_2}, ...\Bigr) - \Bigl(\frac{y_1}{\beta^j_1}, \frac{y_2}{\beta^j_2}, ...\Bigr) \Bigr\|_X^{\alpha}\Bigr) + 2s^n \|u_s\|_{\infty} \nonumber \\
&\leq \Bigl(\sum_{j = 1}^n s^{j-1}\mathrm{Lip}_{A, D_X^{\alpha}} (d_{j})^{-\alpha} \|x - y\|_X^{\alpha}\Bigr) + 2s^n \|u_s\|_{\infty} \nonumber \\
&\leq \Bigl(\sum_{j = 1}^n \mathrm{Lip}_{A, D_X^{\alpha}} (d_{j})^{-\alpha} \|x - y\|_X^{\alpha}\Bigr) + 2s^n \|u_s\|_{\infty} \nonumber \;.
\end{align}
Therefore, taking the limit when $n \to \infty$, it follows that for all $s \in (0, 1)$ and each pair $x, y \in X$,
\begin{equation}
\label{Holder-fixed-point}
|u_s(x) - u_s(y)| \leq d\;\,\mathrm{Lip}_{A, D_X^{\alpha}} \|x - y\|^{\alpha}_X \;,
\end{equation}
where $d$ is as in \eqref{d}.
By the above, it follows that the family $(u_s)_{0<s<1}$ is equicontinuous. Moreover, this implies immediately that $(u^*_s)_{0<s<1}$  (such as defined above)  is equicontinuous as well.

Besides that,
\begin{align}
\mathcal{S}_n(x, y)
&\leq \sup_{(r_1, ..., r_n)} \Bigl\{\Bigl(\sum_{j = 1}^n s^{j-1}V_j(A) + 2s^n \|u_s\|_{\infty} \Bigr\} \nonumber \\
&\leq \sum_{j = 1}^n V_j(A) + 2s^n \|u_s\|_{\infty}\nonumber \;.
\end{align}

Thus, taking the limit as $n \to \infty$, it follows from \eqref{summable-variation-weighted} and \eqref{n-th-variation-sum} that for any $s \in (0, 1)$  is satisfied
\[
|u_s(x) - u_s(y)| \leq V(A) \;.
\]

Since, we have $u^*_s(0) = 0$, for all $s \in (0, 1)$, then, for any $x \in X$ is satisfied
\begin{equation}
|u^*_s(x)| \leq |u_s(x) - u_s(0)| \leq V(A) \;,
\label{uniform-bound}
\end{equation}
that is, the family $(u^*_s)_{0<s<1}$ is uniformly bounded.

Note that for all $x$ and $s$ we get
$$
-\|A\|_{\infty} + s \min u_s \leq u_s(x) \leq \|A\|_{\infty} + s \max u_s \;.
$$

From this, it follows:
\[
-\|A\|_{\infty} \leq (1-s) \min u_s\leq (1-s) \max u_s \leq \|A\|_{\infty} \;,
\]
for any $0<s<1$.

The family $(1 - s)\; u_{s}(0)$, $0<s<1$, is bounded and determines a convergent sequence $s_n$, such that, $\kappa := \displaystyle \lim_{n \to \infty}(1 - s_n)\;\,u_{s_n}(0)$.

Considering the sequences $u_{s_n}$ and $u^*_{s_n}= u_{s_n}- u_{s_n}(0)$, we use Arzel\'a-Ascoli's Theorem in order to get another subsequence (of the given  sequence $u_{s_n}^*$) which  converges uniformly on compact subsets of $X$ and pointwise for all $x \in X$. We also denote this new subsequence of index by $s_n$, $n \in \mathbb{N}$.

Let $u$ be the function satisfying the uniform limit $\displaystyle \lim_{n \to \infty} u^*_{s_n} = u$ on compact subsets of $X$ and the pointwise limit $\displaystyle \lim_{n \to \infty}u^*_{s_n}(x) = u(x)$, for all $x \in X$. Then, for each $x \in X$ the value $u$ satisfies the equation
\begin{equation}
e^{u(x)} = e^{-\kappa}\int_{Lv = x}e^{A(v) + u(v)}dm(v) \;,
\label{eigenfunction}
\end{equation}
with $\kappa = \displaystyle \lim_{n \to \infty}(1 - s_n)u_{s_n}(0)$.

Indeed, since $\mathcal{T}_{s_n, A}(u_{s_n}) = u_{s_n}$, we have
\begin{equation}
\label{fixed-point}
e^{u^*_{s_n}(x)} = e^{-(1 - s_n)u_{s_n}(0)}\int_{Lv = x}e^{A(v) + s_n u^*_{s_n}(v)}dm(v) \;.
\end{equation}

It follows that
$$
\int_{Lv = x}e^{A(v) + s_nu^*_{s_n}(v)}dm(v)
= \int_{\mathbb{R}}e^{A\bigl(r, \frac{x_1}{\alpha_1}, \frac{x_2}{\alpha_2}, ...\bigr) + s_nu^*_{s_n}\bigl(r, \frac{x_1}{\alpha_1}, \frac{x_2}{\alpha_2}, ...\bigr)}dm(r), \nonumber
$$
and by \eqref{uniform-bound} we have that
$$
\int_{\mathbb{R}} \Bigl| e^{A\bigl(r, \frac{x_1}{\alpha_1}, \frac{x_2}{\alpha_2}, ...\bigr) + s_nu^*_{s_n}\bigl(r, \frac{x_1}{\alpha_1}, \frac{x_2}{\alpha_2}, ...\bigr)} \Bigr| dm(r)
 \leq  C \int_{\mathbb{R}} dm(r) < \infty \;,
$$
where $C = e^{\|A\|_{\infty} + V(A)}$. Therefore, \eqref{eigenfunction} is a consequence of  \eqref{fixed-point} and the  Dominated Convergence Theorem.

The foregoing implies that the positive function $\psi_A(x) = e^{u(x)}$, satisfies the equation $\mathcal{L}_A(\psi_A)(x) = \lambda_A\psi_A(x)$ for all $x \in X$, with $\lambda_A = e^{-\kappa}$. Furthermore, by \eqref{Holder-fixed-point}, it follows  for each $x, y \in X$ that
\begin{equation}
\label{Holder-u}
|u(x) - u(y)| = \lim_{n \to \infty}|u^*_{s_n}(x) - u^*_{s_n}(y)| \leq d\,\mathrm{Lip}_{A, D_X^{\alpha}}\|x - y\|^{\alpha}_X \;,
\end{equation}
and by \eqref{uniform-bound}, we obtain that, for any $x \in X$,
\begin{equation}
\label{uniform-bound-u}
|u(x)| = \lim_{n \to \infty}|u^*_{s_n}(x)| \leq V(A) \;,
\end{equation}
By \eqref{Holder-u} and \eqref{uniform-bound-u}, it follows that $u \in \mathrm{Lip}_b(D_X^{\alpha}, X)$, and thus $\psi_A = e^u \in \mathrm{Lip}_b(D_X^{\alpha}, X)$.
In order to prove that $\lambda_A$ is the maximal eigenvalue, we show that $\lambda_A$ coincides with the spectral radius $\sigma$ of the action of $\mathcal{L}_{A} $ on $\mathrm{Lip}_b$. In order to do so, we consider 
\[ R(\varphi) :=  \limsup_{n \to \infty} \left(\|\mathcal{L}_{A}^n (\varphi)\|_\infty + \mathrm{Lip}_{\mathcal{L}^n_A(\varphi), D_X^{\alpha}}  \right)^{1/n}, 
\]
for $\varphi \in \mathrm{Lip}_b(D_X^{\alpha}, X)$. Observe that 
$R(\psi_A) = \lambda_A$ and 
that $ R(\varphi) + t \psi_A) = \max\{ R(\varphi), R(\psi_A)\}$ for any $t > 0$, which gives rise to the following argument. 
Assume that $\sigma > \lambda_A$.
It then follows from the spectral radius formula that  there exists $\varphi \in \mathrm{Lip}_b(D_X^{\alpha}, X)$ with $R(\varphi) > \lambda_A$. As $\inf_x \psi_A(x) > 0$, the above implies that there is $t> 0$ and $c> 0$ such that $\varphi_t:= \varphi + t\psi_A$ satisfies  
$\psi_A/c \leq \varphi_t \leq c\psi_A $ and $R(\varphi_t) = R(\varphi) > \lambda_A$. Hence, for all $n \in \mathbb{N}$,
\[ 0 < \inf_x \psi_A(x) /c \leq \psi_A / c  \leq \mathcal{L}^n_A(\varphi_t) / \lambda_A^n \leq c \psi_A \leq c \sup_x \psi_A(x) < \infty .  \]
As the same argument shows that $\mathcal{L}^n_A(1) / \lambda_A^n$
is uniformly bounded from above and below, an application of item b) of Lemma \ref{lemma-Holder} shows that the local Hölder coefficients of $\mathcal{L}^n_A(\varphi_t) / \lambda_A^n$ are uniformly bounded from above. Hence, $\mathcal{L}^n_A(\varphi_t) / \lambda_A^n$ is a  
Hölder contiunous function with unifomly bounded Lipschitz coefficients (see Lemma \ref{pio}). This implies that 
\[ \lambda_A < R(\varphi_t) = \limsup_{n \to \infty} \left(\|\mathcal{L}_{A}^n (\varphi_t)\|_\infty \right)^{1/n} \leq \limsup_{n \to \infty} \left(\|c \mathcal{L}_{A}^n (\psi_A)\|_\infty \right)^{1/n}  = \lambda_A, \]
which is absurd. Hence, $\lambda_A = \sigma$. 

It remains to show that $\psi_A$ is the unique, strictly positive eigenfunction. In order to do so, one uses $\psi_A$ in order to define the normalization $\overline{A}$ of $A$ as given in Remark \ref{remark:normalization}. The mixing property of item b) then implies that the unique positive eigenfunctions of $\mathcal{L}_{A}^n$ in $\mathrm{Lip}_b(D_X^{\alpha}, X)$ are the constant functions. This concludes the proof of a).

\medskip

\noindent \textit{Part b)}
The proof that there exist a $\sigma$-additive probability $\mu_A$ on $X$ which is fixed for the  action of the operator $\mathcal{L}_A^*$ and that $(\mathcal{L}_A^n)^\ast(\nu) \to \mu_A$ for any $\sigma$-additive probability measure $\nu$  
will be proven in the Appendix (see Section \ref{ap}).
Moreover, note that, for any  $\varphi \in \mathcal{C}_b(X)$, we have
\[
\int_{X}(\varphi \circ L) d\mu_A = \int_{X}(\varphi \circ L) d\left(\mathcal{L}^*_{{A}}(\mu_A)\right) = \int_{X}\mathcal{L}_{{A}}(\varphi \circ L) d\mu_A = \int_{X}\varphi d\mu_A \;.
\]
Hence, $\mu_A$ is $L$-invariant and, as a consequence of Poincar\'e's Recurrence Theorem, it is conservative.  We now show that $\mu_A$ is mixing. As $(\mathcal{L}_A^n)^\ast(\nu) \to \mu_A$ for any $\nu$ in the weak-$\ast$ topology, it follows that, for any $x\in X$ and any bounded continuous function $f: X \to \mathbb{R}$, that  
\[  \mathcal{L}_A^n(f)(x) = \int_X f d (\mathcal{L}_A^n)^\ast(\delta_x)  \xrightarrow{n \to \infty} \int_X f d\mu_A.  \] 
Hence, by Dominated Convergence Theorem, for any  pair of bounded continuous functions $f,g: X \to \mathbb{R}$,
\begin{align*}
\lim_{n\to \infty}  \int_X f\;  g\circ L^n d\mu_A & = \lim_{n\to \infty} \int_X f\;  (g\circ L^n ) d (\mathcal{L}_A^n)^\ast(\mu_A)  
= \lim_{n\to \infty} \int_X \mathcal{L}_A^n(f g \circ L^n)  d\mu_A 
\\
& =\lim_{n\to \infty} \int_X \mathcal{L}_A^n(f) g d\mu_A   = \int_X f d\mu_A \; \int g d\mu_A.  
\end{align*}
This implies that  $\mu_A$ is mixing showing that  all claims of item b) are true.

\smallskip
\noindent \textit{Part c)} This an immediate consequence of parts a) and b), as any potential $A \in \mathrm{Lip}_b(D^\alpha_X, X) \cap \mathcal{SV}(X)$ satisfies $\mathcal{L}_{\overline{A}}(1) = 1$. Indeed, by part a), for each $x \in X$ we have
\begin{align*}
\mathcal{L}_{\overline{A}}(1)(x) 
&= \int_{\mathbb{R}} e^{\overline{A}\bigl(r, \frac{x_1}{\alpha_1}, \frac{x_2}{\alpha_2}, ...\bigr)} d m(r) \\
&= \frac{1}{\lambda_A \psi_A(x)} \int_{\mathbb{R}} e^{A\bigl(r, \frac{x_1}{\alpha_1}, \frac{x_2}{\alpha_2}, ...\bigr)}\psi_A\bigl(r, \frac{x_1}{\alpha_1}, \frac{x_2}{\alpha_2}, ...\bigr) d m(r) \\
&= \frac{\mathcal{L}_A(\psi_A)(x)}{\lambda_A \psi_A(x)} \\
&= 1 \;.
\end{align*}

\medskip

\noindent \textit{Part d)}
Consider $U(x; \epsilon) = \{y \in X ,\; |y_1 - x_1| \leq \epsilon\,, \text{ and } \,x_n = y_n \,; \forall n > 1 \}$. Let $U$ be a neighborhood of $x$, then, there is $\epsilon > 0$ such that $U(x, \epsilon) \subset U$. By Urysohn's Lemma, let $\varphi \in \mathcal{C}_b(X)$ such that $\chi_{U(x; \epsilon)} \leq \varphi \leq U$. Then,
\begin{align*}
\mu_A(U)
&\geq \int_{X} \varphi d\mu_A \\
&= \int_{X}\mathcal{L}_{{A}}(\varphi)d\mu_A \\
&= \int_{X} \Bigl(\int_{\mathbb{R}}e^{{A}\bigl(r, \frac{y_1}{\alpha_1}, \frac{y_2}{\alpha_2}, ...\bigr)}\varphi\bigl(r, \frac{y_1}{\alpha_1}, \frac{y_2}{\alpha_2}, ...\bigr)dm(r)\Bigr)d\mu_A(y) \\
&\geq e^{\inf({A})}\int_{X}\Bigl(\int_{\mathbb{R}}\chi_{U(x; \epsilon)}\bigl(r, \frac{y_1}{\alpha_1}, \frac{y_2}{\alpha_2}, ...\bigr)dm(r)\Bigr)d\mu_A(y) \\
&\geq e^{\inf({A})}m([x_1 - \epsilon, x_1 + \epsilon]) > 0 \;.
\end{align*} 
The above implies that $\mu_A(U) > 0$, for any open neighborhood $U \subset X$ of $x$. Furthermore, since $\mu_A(U) \geq e^{\inf({A})}m([x_1 - \epsilon, x_1 + \epsilon])$, we also have that the support of $\mu_A$ contains the inverse image  of the support of the a priori measure $m$ by the map $\pi_1$. This ends the proof. 
\end{proof}

From Theorem \ref{RPF} and Remark \ref{pq} we can deduce the following result.

\begin{corollary}

Let $1 \leq p  <q$ be real numbers  and $X = l^q(\mathbb{R})$. Assume that $A \in \mathrm{Lip}_b(D_X^{\alpha}, X)$ and $L : l^p(\mathbb{R}) \to l^p(\mathbb{R})$ is a weighted shift satisfying \eqref{d}. Then, Theorem \ref{RPF} holds.
\end{corollary}

\subsection {A class of weighted shifts satisfying Theorem \ref{RPF}}

In this section we will consider a different presentation of some results of last section. We have seen that Theorem \ref{RPF} holds for the class of weighted shifts $L$ satisfying \eqref{d}. Using Cauchy's criterion, we deduce that this class contains  all the weighted shifts satisfying
$$
\lim_{n \to \infty} (d_{n})^{-\frac{1}{n}}=  \lim_{n \to \infty}\Bigl( \inf_{k \geq 1} \beta_k^n \Bigr)^{-\frac{1}{n}} < 1 \;.
$$

In fact, we will prove in the next proposition that the last two classes coincide.

\begin{lemma} (see \cite{BeMe})
\label{rgh}
 The following assertions are equivalent:
\begin{enumerate}
\item $\displaystyle \lim_{n \to \infty}(d_n)^{-\frac{1}{n}} < 1$.
\item $\displaystyle \sum_{n =1}^{\infty}  (d_n)^{-1} <\infty$.
\item $\displaystyle \sup_{k \geq 1} \sum_{n =1}^{\infty} (\beta_{k}^{n})^{-1} < \infty$.
\end{enumerate}
\end{lemma}

\begin{proof}
It is easy to prove  that (i) implies (ii) and  (ii) implies (iii) .
Let us show that (iii) implies (i). For each $k \in \mathbb{N}$, define
$\displaystyle S_k = \sum_{n=1}^{\infty} (\beta_k^{n})^{-1}$. By hypothesis, there is a constant $M > 0$, such that, $S_k \leq M$ for all $k \geq 1$.

Since $\alpha_k^{-1}S_{k+1} = S_k - \alpha_k^{-1}$ for each $k \in \mathbb{N}$, it follows that
\begin{align}
(\beta_k^n)^{-1}
&= \frac{S_k}{1 + S_{k+n}} \cdot \frac{S_{k+1}}{1 + S_{k+1}} \cdot...\cdot \frac{S_{k+n-1}}{1 + S_{k+n-1}} \nonumber \\
&\leq S_k \cdot \frac{S_{k+1}}{1 + S_{k+1}} \cdot...\cdot \frac{S_{k+n-1}}{1 + S_{k+n-1}} \nonumber \\
&\leq M \cdot \Big(\frac{M}{1+M}\Big)^{n-1} \nonumber \;.
\end{align}

This implies that $\displaystyle \limsup_{n \to \infty}(d_n)^{-\frac{1}{n}} \leq \frac{M}{1+M}< 1$.

On the other hand, we have
$$
(d_n)^{-\frac{1}{n}}= \Bigl( \sup_{k \geq 1} (1/ \beta_k^n) \Bigr)^{\frac{1}{n}} \;.
$$

Moreover, by \eqref{spectral-ratio} we have $\displaystyle \lim_{n \to \infty} \Big( \sup_{k \geq 1} (1/ \beta_k^n) \Bigr)^{\frac{1}{n}} = \lim_{n \to \infty}\|(L')^n\|_{op}^{\frac{1}{n}}$, which is the spectral radius $r(L')$ of the weighted shift operator $L'$ defined by
$$
L'((x_n)_{n \geq 1}) = (\alpha_n^{-1} x_{n+1})_{n \geq 1} \;.
$$

Hence $\displaystyle \lim_{n \to \infty}(d_n)^{-\frac{1}{n}} = \limsup_{n \to \infty}(d_n)^{-\frac{1}{n}} < 1$.
\end{proof}

We deduce from Lemma \ref{rgh} that for all  $\alpha > 0,$ the condition $ \sum_{n =1}^{\infty} (d_n)^{-\alpha} < \infty  $ is equivalent to $\displaystyle \lim_{n \to \infty}(d_n)^{-\frac{1}{n}} < 1$.

As a consequence, we obtain the following result.

\begin{proposition}
\label{pr-limsup}
Let $X$ be either $c_0(\mathbb{R})$ or $l^p(\mathbb{R})$, $1 \leq p < \infty$. Assume that $A \in \mathrm{Lip}_b(D_X^{\alpha}, X)$ and $L: X \to X$ is a weighted shift satisfying
 \begin{eqnarray}
 \label{bebeta}
 \displaystyle \lim_{n \to \infty}(d_n)^{-\frac{1}{n}} < 1 \;.
 \end{eqnarray}

Then, the claims of items a), b), c) and d) of Theorem \ref{RPF} are satisfied.
\end{proposition}

\begin{proof}
By Lemma \ref{rgh}, $\displaystyle \lim_{n \to \infty}(d_n)^{-\frac{1}{n}} < 1$ is equivalent to \eqref{d}.
\end{proof}

\begin{remark}
Consider $X \in \{c_0(\mathbb{C}), \; l^p(\mathbb{C}),\; 1 \leq p < \infty\}$. Note that the claims of Theorem \ref{RPF} and Proposition \ref{pr-limsup} hold if $A \in \mathrm{Lip}_b(D_X^{\alpha}, X)$ and $L: X \to X$ is a weighted shift satisfying  $\displaystyle \lim_{n \to \infty}\Bigl( \inf_{k \geq 1} \vert \beta_k^n \vert \Bigr)^{-\frac{1}{n}} < 1$.
\end{remark}



\begin{example}
There are many cases which can be considered on the class of weighted shifts operators satisfying \eqref{bebeta}. For instance:
\begin{enumerate}
\item
Take
$$
1 < c <\alpha_n \leq c', \mbox { for all } n \geq n_0 \;,
$$
where $n_0$ is a non-negative integer and $c$ is a fixed real number.
\item
Assume that there exists an increasing sequence of non-negative integers $(k_n)_{n \geq 0}$, such that, $k_0=0$ and $k_i < k_{i+1}$, for all integer $i \geq 0$.
Let $a$ and $ b$ to be two real numbers, such that, $c \leq a <1< b \leq c'$,
and
$$
a \leq \alpha_j <1, \mbox { for all } k_{2i}+1  \leq j \leq  k_{2i+1},\; i \geq 0
$$
and
$$
1< b \leq \alpha_j <c', \mbox { for all } k_{2i+1}+1  \leq j \leq  k_{2i+2},\; i \geq 0 \;.
$$

Set
$$
r_n= k_{n}- k_{n-1} \mbox { for all } n \geq 1 \;.
$$

Assume that there exists a real number $e \geq 1$, such that,
$$
1 \leq r_n \leq e,  \mbox { for all } n \geq 1 \;.
$$

Then,
$$
\beta_k^n= \alpha_{k} ... \alpha_{k+n-1} \geq (a^e b)^{[\frac{n} { e+1}]} a^{e}, \mbox { for all } k \geq 1 \;,
$$
where $[\frac{n} { e+1}]$ is the integer part of $\frac{n} { e+1}$. Assuming that $ a^e b >1$ we are able to present an interesting example.
\end{enumerate}
\end{example}



\section{Extension to a wider class of operators }
\label{extension}

Let $X$ be a Banach space and $T: X \to X$ be a linear continuous operator.
We denote by $m(T)$ the co-norm of $T$; that is, $m(T)= \inf \{\|T(x)\|_X,\; \|x\|_X = 1\}$.

It is known that $m(T)>0$, if and only if, $T$ is one to one and has closed range. Moreover, it is easy to see that for all $x \in X,\; m(T) \|x\|_X \leq \|T(x)\|_X$. Another important property is that
$\displaystyle \lim_{n \to \infty} (m(T^n))^{\frac{1}{n}}$ exists (see \cite{Bo}).

Now assume that $T$ is onto, not necessarily one to one. Suppose also  that there exists   a closed subspace $E$ of $X$, such that,
$X$ is the direct sum of $\mathrm{Ker} (T)$ and $E$, i.e, any vector $x \in X$ can be written in a unique form as $x = z + v$, with $z \in \mathrm{Ker} (T)$ and $v \in E$. We will use the notation
\begin{eqnarray}
\label{sjn}
 X= \mathrm{Ker}(T) \oplus E \;.
 \end{eqnarray}
 We define
$$
p(T) = \inf \{\|T(x)\|_X,\; \|x\|_X = 1,\; x  \in E\} \;.
$$

Note that
\begin{eqnarray}
\label{fg}
p(T) \|x\|_X \leq \|T(x)\|_X, \mbox { for all } x \in E.
\end{eqnarray}


We point out that in case the Banach space we consider in this section is a Hilbert space, then 
the required  assumptions for splittings are automatically true.

\begin{example} We consider the weighted shift $L: X  \to X$, where $X$ is $c_0(\mathbb{R})$, or $l^p(\mathbb{R}), \; 1 \leq p < \infty$,  which was given by
$$
L(x_1, x_2, ...)= (\alpha_1 x_2, \alpha_2 x_3,...) \;.
$$
Then, we have
$$
\mathrm{Ker} (L)= \{(x_i)_{i \geq 1} \in X ,\; x_i= 0 \mbox { for all } i \geq 2\}
$$
and
$$
E= \{(x_i)_{i \geq 1} \in X,\; x_1= 0 \} \;.
$$

It's not difficult to see that
\begin{eqnarray}
\label{lgn}
p(L)= \inf_{k \geq 1} \alpha_k.
\end{eqnarray}

Indeed, for all $n \geq 2$, we have $\|L(e_n)\|_X = \| \alpha_{n-1} e_{n-1} \|_X = \alpha_{n-1}$.
Hence
$$
p(L) \leq \inf_{k \geq 1} \alpha_k \;.
$$
On the other hand,
$$
\|L((x_n)_{n \geq 1} \| \geq (\inf_{k \geq 1} \alpha_k) \; \|(x_{n+1})_{n \geq 1} \|, \mbox { for all } (x_2,x_3 ... ) \in X \;.
$$
Hence $ p(L) \geq \inf_{k \geq 1} \alpha_k$ and we obtain \eqref{lgn}.

Now fixing an integer $n \geq 1$, we get
$ X= \mathrm{Ker}(L^n) \oplus E_n,$
where
$$
\mathrm{Ker} (L^n)= \{(x_i)_{i \geq 1} \in X ,\; x_i= 0 \mbox { for all } i \geq n+1\}
$$
and
$$
E_n= \{(x_i)_{i \geq 1} \in X,\; x_i= 0, \forall \; 1 \leq i \leq n \} \;.
$$

In a similar way as in the case $n=1$, we obtain
$$
\displaystyle p(L^n)= d_n = \inf_{k \geq 1} \beta_k^n \;.
$$

Observe that  $\displaystyle \lim_{n \to \infty} (p(L^n))^{\frac{1}{n}}$ is
the inverse of the constant given in Proposition \ref{pr-limsup}.
\end{example}

The above claim can be stated  in a more general form. For instance, assume that $T: X \to X$ is onto, not one to one and satisfies 
\begin{eqnarray}
\label{sjnf1}
X = \mathrm{Ker}(T^n) \oplus E_n  \mbox { and } T(E_{n+1}) = E_n,\mbox { for all } n \geq 1,
\end{eqnarray} 
where the $E_n$'s are closed subspaces of $X$.
We also assume that $T(E_1)= X.$

Note that the assumption \eqref{sjnf1} is satisfied by the weighted shift $L$ in any space $X \in \{c_{0}(\mathbb{R}),\; l^p(\mathbb{R}),\; 1 \leq p < \infty\}$. Furthermore, \eqref{sjnf1} is also satisfied by any continuous operator $T: X \to X$  onto, but not one to one, where $X$ is a Hilbert space, since in this case, for all integer $n \geq 1$, the set $E_n$ equals the orthogonal space of $\mathrm{Ker} (T^n)$.

Now, suppose that $p(T^n)>0$ for all $n \geq 1$,
then $\displaystyle \lim_{n \to \infty} (p(T^n))^{-\frac{1}{n}}$ exists.

Indeed, we have
$$
p (T^{n+m}) \geq p(T^n) p(T^m), \mbox { for all } n , m \in \mathbb{N} \;.
$$
Hence, the sequence $(-\log(p(T^n)))_{n \geq 1}$ is sub-additive.
Thus, $\displaystyle \lim_{n \to \infty}  -\frac{\log(p(T^n))}{n}$ exists and is equal to $\displaystyle \inf_{n \geq 1} -\frac{\log (p(T^n))}{n}$.

Now, for all $x \in X \setminus \{0\}$, denote
$$
T^{-1}(\{x\}) := \{v \in X,\; T(v)= x\} \;.
$$

Since $T$ is onto, $T^{-1}(\{x\})$ is not empty and
$$
T^{-1}(\{x\})= \{x'\} + \mathrm{Ker}(T)= \{x'+ z,\; z \in \mathrm{Ker} (T)\} \;,
$$
where $x'$ is an arbitrary element of $T^{-1}(\{x\})$.

We now fix an {\it a priori} probability $m$ on $\hbox{Ker}(T)$. We will also denote by $m$ the corresponding {\it a priori} measure on $T^{-1}(\{x\}) = \{x'\} + \mathrm{Ker}(T)$.
Then, this gives rise to a  Ruelle operator associated to the potential $A \in \mathrm{Lip}_b(D_X^{\alpha}, X)$ and the a priori probability $m$, defined as the map assigning to each $\varphi \in \mathcal{C}_b(X))$ the function 
$$
\mathcal{L}_A(\varphi)(x) := \int_{v  \in T^{-1}(\{x\})} e^{A(v)}\varphi(v) dm(v) = \int_{z \in \mathrm{Ker}(T)} e^{A(x'+ z)}\varphi(x'+ z) dm(z) \;.
$$

Some of the results in this section follow the same reasoning as before, and sometimes we just outline the proof. 

\begin{lemma}
Let $X$ be a Banach space and $T: X \to X$ a bounded linear operator, such that, $T$ is onto, not one to one, satisfying \eqref{sjn}, $p(T)>0$ and $A \in \mathrm{Lip}_b(D_X^{\alpha}, X)$. Then, the Ruelle operator $\mathcal{L}_A$ preserves the spaces $\mathrm{Lip}_b(D_X^{\alpha}, X)$ and $\mathrm{Lip}^{loc(\delta)}_b(D_X^{\alpha}, X)$, $\delta > 0$.
\end{lemma}

\begin{remark}
It's classical result (and not difficult to prove) that  $p(T)>0$ is equivalent to the fact that $T: E \to X$ is injective and  $T(E)$ is a closed subspace of $X$.
\end{remark}

\begin{proof}
For any $\varphi \in \mathrm{Lip}_b(D_X^{\alpha}, X)$ and each $x, y \in X$, we have

\begin{align}
&\left| \mathcal{L}_A(\varphi)(x) - \mathcal{L}_A(\varphi)(y)\right| \nonumber \\
&= \left| \int_{z \in \mathrm{Ker}(T)} e^{A(x'+ z)}\varphi(x'+ z) dm(z) - \int_{z \in \mathrm{Ker}(T)} e^{A(y'+ z)}\varphi(y'+ z) dm(z)\right|, \nonumber
\end{align}
where 
$$
T^{-1}(\{x\})= \{x'\} + \mathrm{Ker}(T),\; T^{-1}(\{y\})= \{y'\} + \mathrm{Ker}(T) \;.
$$
Hence,
\begin{align}
& \left| \mathcal{L}_A(\varphi)(x) - \mathcal{L}_A(\varphi)(y) \right| \nonumber \\
& \leq   \int_{\mathrm{Ker}(T)} e^{\|A\|_{\infty}} \left| \varphi(x'+ z)  - \varphi(y'+ z)  \right| \nonumber 
+ \|\varphi\|_{\infty}\left|e^{A(x'+ z)} - e^{A(y'+ z)}  \right| dm(z) \nonumber \\
&\leq \left( e^{\|A\|_{\infty}} \mathrm{Lip}_{\varphi, D_X^{\alpha}} + \|\varphi\|_{\infty} \mathrm{Lip}_{e^A, D_X^{\alpha}} \right) \int_{\mathrm{Ker}(T)} \|x'- y'\|_X^{\alpha}dm(z) \;. \nonumber
\end{align}

Using \eqref{fg} we deduce that
$$
\left| \mathcal{L}_A(\varphi)(x) - \mathcal{L}_A(\varphi)(y) \right|
 \leq  \frac{ e^{\|A\|_{\infty}} \mathrm{Lip}_{\varphi, D_X^{\alpha}} + \|\varphi\|_{\infty} \mathrm{Lip}_{e^A, D_X^{\alpha}} }{p(T)^{\alpha}}  \|x- y\|_X^{\alpha} \;.
$$

Then, $\mathcal{L}_A(\varphi) \in \mathrm{Lip}_b(D_X^{\alpha}, X)$. Under the assumption $\|x - y\|^{\alpha}_X < \delta$ the proof for the case $\mathrm{Lip}^{loc(\delta)}_b(D_X^{\alpha}, X)$ follows the same reasoning.
\end{proof}

We say that a potential $A \in \mathcal{C}(X)$ has {\it summable variation} with respect to the linear operator $T : X \to X$ if satisfies
\begin{equation}
\label{summable-variation}
V_T(A) = \sum_{n = 1}^{\infty} V_{T,n}(A) < \infty \;,
\end{equation}
where
\[
\label{summable-variation1}
V_{T,n}(A) := \sup\{|A(z_n + x_n) - A(z_n + y_n)| :\; z_n \in \mathrm{Ker}(T^n),\; x_n, y_n \in E_n\}.
\]

We will denote by $\mathcal{SV}_T(X)$ the set of potentials $A \in \mathcal{C}(X)$ satisfying \eqref{summable-variation}.

When considering a more general setting (as in this section) we need an assumption  that is stronger that 
adapted tails condition which was used before.

\begin{definition}  \label{Def:strong-adapted}
We say that $m$ has {\it strong adapted tails} if, for any $\epsilon > 0$ there exists a sequence of positive numbers $(\kappa_n)_{n \geq 1}$, such that, 
\begin{enumerate}
\item $\sum_{n=1}^\infty m(\mathbb{R} \setminus [-p(T^n)\kappa_n, p(T^n)\kappa_n]) < \epsilon$;
\item the sequence $ (\kappa_n)_{n \geq 1} $ is in $l^1(\mathbb{R})$. 
\end{enumerate}
\end{definition}

\begin{theorem}
\label{RPF-separable}
Let $X$ be a separable Banach space and $T: X \to X$ a bounded linear operator, such that, $T$ is onto, not bijective, satisfying \eqref{sjnf1}, $p(T^n)>0$  and $\dim(\mathrm{Ker}(T^n)) < \infty$, for each $n \in \mathbb{N}$.  Consider $A \in \mathrm{Lip}_b(D_X^{\alpha}, X)$, assume that the a priori measure $m$ has strong adapted tails and 
$$
\sum_{n=1}^{  \infty} (p (T^n))^{-\alpha} < \infty \;.
$$

Then:
\begin{enumerate}[a)] 
\item
If $A \in \mathcal{SV}_T(X) $, then there exist $\lambda_A > 0$ and  a strictly positive function $\psi_A \in \mathrm{Lip}_b(D_X^{\alpha}, X)$  such that $\mathcal{L}_A (\psi_A)(x) = \lambda_A \psi_A(x)$, for each $x \in X$.
\item If $m$ has adapted tails and $\mathcal{L}_A(1) = 1$, then there exists a unique $\sigma$-additive Borel probability measure $\mu_A$ which is a fixed point for the operator $\mathcal{L}^*_{A}$. Moreover, $\lim_{n \to \infty} (\mathcal{L}^n_{A})^\ast(\nu) = \nu$ in the weak-$\ast$ topology for any Borel probability measure $\nu$. In particular, $\mu_A$ is an invariant and mixing probability measure (and therefore conservative and ergodic).
\item If $A \in \mathcal{SV}_T(X)$ and  $m$ has adapted tails, then  there exists a unique fixed point $\mu_A$  for the operator $\mathcal{L}^*_{\overline{A}}$, where   $\overline{A} = A + \log(\psi_A) - \log(\psi_A \circ L) - \log(\lambda_A)$. Moreover,  $\mu_A$ is an invariant and mixing $\sigma$-additive probability measure.
\item If $A \in \mathcal{SV}_T(X)$ and $\pi : X \to \mathrm{Ker}(T)$ is  the projection on the kernel of $T$, we get 
$
 \hbox{supp}(\mu_A)  \supset   \pi^{-1}(\hbox{supp}(m)) 
 $. From this follows that $\hbox{supp}(\mu_A) = X$.
\end{enumerate}
\end{theorem}

\begin{proof}
The proof follows  basically  the same reasoning of  the proof of Theorem \ref{RPF}.

\noindent \textit{Part a)} For each $s \in (0, 1)$, we consider  the operator $\mathcal{T}_{s, A} : \mathcal{C}_b(X) \to \mathcal{C}_b(X)$ given by 
\begin{equation}
\mathcal{T}_{s, A}(u)(x) = \log\left( \int_{v  \in T^{-1}(\{x\})} e^{A(v) + su(v)} dm(v) \right) \nonumber \;.
\label{discount3}
\end{equation}
In the same way as before, for all $s \in (0, 1)$, the map $\mathcal{T}_{s, A}$ has a unique fixed point $u_s \in \mathcal{C}_b(X)$. That is, for all $x$ and $0 < s < 1$ 
$$
e^{u_s(x)} = \int_{v \in T^{-1}(\{x\}) } e^{A(v) + su_s(v)} dm(v) \;.
$$
Following reasoning of the proof of Theorem \ref{RPF} we need to prove that the family of bounded continuous functions is equicontinuous $(u_{s})_{0< s<1}$.

In the same way as was done before, we have
$|u_s(x') - u_s(y)| \leq \mathcal{S}_1(x, y)$, where
$$
\mathcal{S}_1(x, y)= \sup_{z \in \mathrm{Ker}(T)} \Bigl\{A( x'+ z)- A(y'+z)+ s( u_s (x'+ z)- u_s (y'+ z))\Bigr\} \;.
$$

For elements  $x', y'$ in $X$, such that, $T(x')= x$ and $T(y')=y$, take
\begin{align}
\mathcal{S}_n(x, y)
&=  \sup_{\substack{z_j \in \mathrm{Ker}(T^j) \\ 1 \leq j \leq n}} \Bigl\{\Bigl(\sum_{j = 1}^n s^{j-1} \bigl(A(x_j+ z_j) - A(y_j+ z_j)\bigr)\Bigr)
 \nonumber \\
& \ \  \ \  \ \  \ \  \ \  \ \  \ \
+ s^n (u_s(x_n+ z_n) - u_s(y_n+ z_n)) \Bigr \} \nonumber \;,
\end{align}
where $x_j$ and $y_j,\; 1 \leq j \leq n$, are fixed elements of $X$ satisfying
$$T^j(x_j)= x, \; T^{j}(y_j)=y \mbox { for all } 1 \leq j \leq n \;.$$

The sequence  $(\mathcal{S}_n(x, y))_{n \geq 1}$ satisfies
\begin{align}
\mathcal{S}_n(x, y)
&\leq \sup_{\substack{z_j \in \mathrm{Ker}(T^j) \\ 1 \leq j \leq n}} \Bigl\{\Bigl(\sum_{j = 1}^n s^{j-1}\mathrm{Lip}_{A, D_X^{\alpha}} \|x_j - y_j \|_X^{\alpha}
+ 2s^n \|u_s\|_{\infty} \Bigr\} \nonumber \\
&\leq \Bigl(\sum_{j = 1}^n s^{j-1}\mathrm{Lip}_{A, D_X^{\alpha}} \frac{1}{p(T^{j})^{\alpha}} \|x - y\|_X^{\alpha}\Bigr) + 2s^n \|u_s\|_{\infty} \nonumber \\
&\leq \Bigl(\sum_{j = 1}^n \mathrm{Lip}_{A, D_X^{\alpha}} \frac{1}{p(T^{j})^{\alpha}} \|x - y\|_X^{\alpha}\Bigr) + 2s^n \|u_s\|_{\infty} \nonumber \;.
\end{align}

We deduce that for all $s \in (0, 1)$ and each pair $x, y \in X$ is satisfied

$$
|u_s(x) - u_s(y)| \leq \Bigl( \mathrm{Lip}_{A, D_X^{\alpha}} \sum_{n = 1}^{\infty} \frac{1}{p(T^{n})^{\alpha}}\Bigr) \; \|x - y\|_X^{\alpha} \;.
$$

Besides that, by \eqref{summable-variation}, it follows that
\[
|u_s(x) - u_s(y)| \leq V_T(A) \,.
\]

This is the end of the proof of item a).
\medskip

\noindent \textit{Part b)} This part follows exactly the same reasoning followed in the  proof of Theorem \ref{RPF} and will not be presented.

\medskip

\noindent \textit{Part c)} This part is a consequence of a) and b). Indeed,  for any $A \in \mathrm{Lip}_b(D^\alpha_X, X) \cap \mathcal{SV}_T(X)$ and any $x \in X$ we have
\begin{align*}
\mathcal{L}_{\overline{A}}(1)(x) 
&= \int_{\mathrm{Ker}(T)} e^{\overline{A}(x' + z)} d m(z), \ \  \ \ x' \in T^{-1}(\{x\}) \\
&= \frac{1}{\lambda_A \psi_A(x)} \int_{\mathrm{Ker}(T)} e^{A(x' + z)}\psi_A(x' + z) d m(z), \ \  \ \ x' \in T^{-1}(\{x\}) \\
&= \frac{\mathcal{L}_A(\psi_A)(x)}{\lambda_A \psi_A(x)} \\
&= 1 \;.
\end{align*}

\medskip

\noindent \textit{Part d)} Define $U(x; \epsilon) := \{y \in X ,\; |\pi(y) - \pi(x)| \leq \epsilon\,, \text{ and } y - \pi(y) = x - \pi(x) \}$, and let $U$ be a neighborhood of $x$. There exists $\epsilon > 0$ such that the closed set $U(x; \epsilon)$ is contained in $U$. By Urysohn's Lemma, let $\varphi \in \mathcal{C}_b(X)$ such that $\chi_{U(x; \epsilon)} \leq \varphi \leq \chi_U$. Then, 
\begin{align*}
\mu_A(U)
&\geq \int_{X}\varphi d\mu_A \nonumber \\
&= \int_{X}\mathcal{L}_{{A}}(\varphi)d\mu_A \nonumber \\
&= \int_{X} \Bigl(\int_{\mathrm{Ker}(T)} e^{A(y' + z)} \varphi(y' + z)dm(z)\Bigr)d\mu_A(y), \ \  \ \  y' \in T^{-1}(\{y\}) \\
&\geq e^{\inf({A})}\int_{X}\Bigl(\int_{\mathrm{Ker}(T)}\chi_{U(x; \epsilon)}\bigl(y' + z) dm(z)\Bigr)d\mu_A(y), \ \  \ \  y' \in T^{-1}(\{y\}) \\
&\geq e^{\inf({A})}m(\epsilon B_{\mathrm{Ker}(T)} + \{\pi(x)\}) > 0 \nonumber \;.
\end{align*}

The above implies that $\mu_A(U) > 0$, for any open neighborhood $U \subset X$ of $x$. moreover, since $\mu_A(U) \geq e^{\inf({A})}m(\epsilon B_{\mathrm{Ker}(T)} + \{\pi(x)\})$, we have $\mathrm{supp}(\mu_A) \supset \pi^{-1}(\hbox{supp}(m))$.
\end{proof}

It follows from  the above that:

\begin{corollary}
\label{corollary-RPF-separable}
Consider $X$ a separable Banach space and $T: X \to X$ a bounded linear operator such that $T$ is onto, not injective,  satisfying \eqref{sjnf1}, $p(T^n)>0$ and $\dim(\mathrm{Ker}(T^n)) < \infty$, for each $n \in \mathbb{N}$.  Consider $A \in \mathcal{SV}_T(X) \cap \mathrm{Lip}_b(D_X^{\alpha}, X)$. Assume that  the {\it  a priori} probability measure $m$ has strong adapted tails and that  the following is true:
$$
\lim_{n \to \infty}  (p (T^n))^{-\frac{1}{n}} < 1 \;.
$$

Then,  items a), b), c) and d) of Theorem \ref{RPF-separable} are true.
\end{corollary}


In the next result, we prove that a sub class of  the class of operators satisfying the claim of  last Corollary  is formed by frequently hypercyclic operators.

\begin{proposition}
\label{hyperc}
Let $X$ be a separable Banach space and $T: X \to X$ be a bounded linear operator, such that, $T$ is onto, not one to one, satisfying \eqref{sjnf1}  and  $\displaystyle \bigcup_{n = 0}^{\infty} \mathrm{Ker}(T^n)$ is dense in $X$.  Assume that
$$
\lim_{n \to \infty}  (p (T^n))^{-\frac{1}{n}} < 1 \;,$$
then,  $T$ is frequently hypercyclic and Devaney chaotic.
\end{proposition}

\begin{example}
If $L: X \to X$ is a weighted shift and $X$ is $c_0(\mathbb{R})$ or $l^p(\mathbb{R})$, $1 \leq p < \infty$. Then,
$$
\bigcup_{n = 0}^{\infty} \mathrm{Ker}(L^n)= \{(x_{n})_{n \geq 1} \in X,\; \exists n_0 \in \mathbb{N},\; x_n=0,\; \forall n \geq n_0\} \;.
$$
Hence, $\displaystyle \bigcup_{n = 0}^{\infty} \mathrm{Ker}(L^n)$ is dense in $X$.
\end{example}

\begin{remark}
There is an efficient criterion which guarantees that $T$ is Devaney  chaotic and frequently hypercyclic (see \cite{BaMa}).

Let $X$ be a separable Banach space and $T: X \to X$ be a continuous linear operator.
 Assume that there
exist a dense set $D \subset X$ and a map $S :D  \to D$, such that,

\begin{enumerate}
\item
For any $x \in D$, the series $\displaystyle \sum_{n =0} ^{\infty} T^{n}(x)$ and $\displaystyle \sum_{n =0} ^{\infty} S^{n}(x)$ are unconditionally convergent (all sub-series of both series are convergent).
\item
For every $x \in D,\;  T \circ S(x)= x$.
\end{enumerate}
Then, $T$ is  Devaney chaotic and frequently hypercyclic.
\end{remark}

\begin{proof}[Proof of Proposition \ref{hyperc}]
Set $\displaystyle D= \bigcup_{n = 0}^{\infty} \mathrm{Ker}(T^n)$, then
for all $x \in D$, there exists a non-negative integer  $n_0 = n_0(x)$, such that,
$$
T^n (x)= 0, \mbox { for all } n \geq n_0 \;.
$$
Hence, the series $\displaystyle \sum_{n =0} ^{\infty} T^{n}(x)$ is unconditionally convergent.

On the other hand, since $T$ is onto, then for any $x \in D \setminus \{0\}$, there exists
$y \in X$ such that $T(y)= x$. Taking $S(x)= y,$  we deduce by induction that
$$T^n \circ S^n (x)= x, \mbox { for all } n \in \mathbb{N}.$$
Hence, by \eqref{fg},  it follows that
$$
\|S^n (x) \|_X \leq \frac{\|x \|_X}{p(T^n)}, \mbox { for all } n\geq 1 \;.
$$

Since $\displaystyle \lim_{n \to \infty}  (p (T^n))^{-\frac{1}{n}} < 1,$
we deduce that the series $\displaystyle \sum_{n =0}^{\infty} S^{n}(x)$ is absolutely convergent and hence unconditionally convergent, and we are done.
\end{proof}




\section{Appendix} \label{ap}

The purpose of this appendix is to prove the existence of a unique fixed point of the operator $\mathcal{L}_A^*$, when $A$ is normalized.
Throughout, we assume that $X =c_0(\mathbb{R})$ or $X= l^p(\mathbb{R}) ,\; 1 \leq p < \infty$, and that $A \in \mathrm{Lip}_b(D_X^{\alpha}, X)$ and $\mathcal{L}_A(1) = 1$. For this setting, our aim is to prove that  the operator $\mathcal{L}_A^*$ has a fixed point Borel probability measure on $X$. In order to do so,  we are following the approach to spectral gaps for Markov operators by M. Hairer and J. C. Mattingly in \cite{HaM}, which was adapted to the dual of the Ruelle 
operator by several authors (see \cite{Sta},\cite{KLS},\cite{Lo},\cite{BS},\cite{CSS},\cite{Elis},\cite{SVZ}, in chronological order). 

\begin{remark} \label{esta1}
However, as a consequence of the metric structure of $X$, the adapted approach does not provide uniform contraction rates and, in particular, we only obtain convergence but not geometric convergence of the iterates of $\mathcal{L}_A^*$ . 
\end{remark}

We now recall the 
strategy of the approach which relies on a change to the bounded and equivalent metric 
\[\widetilde{D}(x, y) :=   \min\left\{1,a D_X^{\alpha}(x, y)\right\}. \]
By choosing $a > 0$ properly, a version of the Doeblin-Fortet inequality then implies that $(\mathcal{L}_A^n)^*$ locally contracts the Wasserstein distance for some $n \in \mathbb{N}$. Moreover, a coupling construction then makes use of the boundedness of $\widetilde{D}$ in order to obtain a contraction for Dirac measures whose base points have $\widetilde{D}$-distance bigger than 1. By joining these two results on a further coupling, one then obtains the desired contraction property.

We begin with a comparison of the spaces of Lipschitz continuous functions with respect to the change of metric. For ease of notation, set $D:= D_X^{\alpha}$. Moreover, recall that $\mathrm{Lip}^{{\mathrm  loc }(\delta)}(D, X)$ is the set of continuous functions $\varphi: X \to \mathbb{R}$ such that
\[
\mathrm{Lip}^{\mathrm{loc}(\delta)} _{\varphi, D} := \sup\left\{\frac{|\varphi(x) - \varphi(y)|}{D(x, y)} :\; x \neq y,\; D(x, y) < \delta \right\} < \infty \;.
\]

\begin{lemma} 
\label{pio} 
A function $\varphi$ is in  $ \mathrm{Lip}(\widetilde{D}, X)$ if and only if $\varphi \in\mathrm{Lip}^{\mathrm{{loc}}({1}/{a})}_b (D,X)$. Moreover, $ a \mathrm{Lip}^{\mathrm{{loc}}({1}/{a})}_{\varphi, D} =  \mathrm{Lip}^{\mathrm{{loc}}(1)}_{\varphi, \widetilde{D}}$, 
\begin{equation} \label{helpx}
 a \mathrm{Lip}^{\mathrm{{loc}}({1}/{a})}_{\varphi, D} \leq \mathrm{Lip}_{\varphi, \widetilde{D}} \leq \max\left\{ 2 \|\varphi\|_\infty,  \frac{1}{a} \,\mathrm{Lip}^{\mathrm{{loc}}({1}/{a})}_{\varphi, D} \right\},  
\end{equation}
and $\sup(\varphi) - \inf(\varphi) \leq \mathrm{Lip}_{\varphi, \widetilde{D}}$.  
 \end{lemma}
 
 \begin{proof}
Note that $D(x, y) < {1}/{a}$ if and only if $\widetilde{D}(x,y) < 1$ which proves that $ a \mathrm{Lip}^{\mathrm{loc}({1}/{a})}_{\varphi, D} =  \mathrm{Lip}^{\mathrm{{loc}}(1)}_{\varphi, \widetilde{D}}$. 

 Now assume that $\varphi \in \mathrm{Lip}^{\mathrm{{loc}}({1}/{a})} (D,X)$ is bounded. Then
  \[ 
\frac{|\varphi(x) - \varphi(y) |}{ \widetilde{D}(x,y)}  = 
\begin{cases} \frac{|\varphi(x) - \varphi(y) |}{ a D(x,y)} \leq  \frac{1}{a} \,\mathrm{Lip}^{\mathrm{{loc}}({1}/{a})}_{\varphi, D} &:\; D(x, y) < {1}/{a}\\
|\varphi(x) - \varphi(y) | \leq 2 \|\varphi\|_\infty &:\; D(x, y) \geq {1}/{a}
\end{cases}
 \]
This proves that $\varphi \in \mathrm{Lip}(\widetilde{D}, X)$ and the left half of \eqref{helpx}. The  left half of  \eqref{helpx}
follows from the trivial estimate $  \mathrm{Lip}^{\mathrm{{loc}}(1)}_{\varphi, \widetilde{D}} \leq  \mathrm{Lip}_{\varphi, \widetilde{D}}$ and  the first observation. Finally, as $\sup \widetilde{D}(x,y)  = {1}$,
\[ \sup(\varphi) - \inf(\varphi)   \leq  \mathrm{Lip}_{\varphi, \widetilde{D}} \; \sup_{x,y \in X} \widetilde{D}(x,y) = \mathrm{Lip}_{\varphi, \widetilde{D}}. \]
\end{proof}

In order to introduce the Wasserstein distance (see for details \cite{Bol, Vi}), recall that a coupling of $\mu,\nu \in \mathcal{P}(X)$ is a  Borel probability measure $\Pi$ on $X \times X$  
such that the first marginal of $\Pi$ is $\mu$ and the second marginal is $\nu$. We refer to $\Gamma(\mu,\nu)$ as the set of all couplings of $\mu$ and $\nu$. The $D$-Wasserstein distance on $\mathcal{P}(X)$ is then defined by 
\begin{equation}
\label{defw} 
W_{D} (\mu,\nu) = \inf \left\{ \int_{X \times X} D(x,y) d \Pi(x,y) : \Pi \in  \Gamma(\mu,\nu) \right\} \nonumber \;,
\end{equation} 
where $D$ is a metric whose open balls generate the Borel $\sigma$-algebra.  Moreover, if $X$ has finite diameter with respect to $D$, then  $W_{D}$ is compatible with the weak convergence of measures, that is, $\mu_n$ converges weakly to $\mu$ if and only of $W_{D} (\mu_n,\mu) \to 0$. A further fundamental tool in this setting is 
Kantorovich's duality, which states that 
 \begin{equation}
\label{DuaK}
W_{D} (\mu,\nu) = \sup \left\{ \Bigl|\int_X \varphi d \mu - \int_X \varphi d \nu \Bigr| :\; \mathrm{Lip}_{\varphi, D}\,\,\leq 1 \right\} \nonumber \;.
 \end{equation} 

We are now in position to prove the basic contraction estimates for our setting. In order to do so, observe that $\mathcal{L}_A(1) =1 $ implies that 
$$
\mathcal{L}_A^*(\mu)(X) = \int_X 1 d\mathcal{L}_A^*(\mu) = \int_X \mathcal{L}_A(1) d\mu =1 \text{ for any } \mu \in  \mathcal{P}(X) \;.
$$

Hence, as 
 $A$ is a normalized potential,  $\mathcal{L}_A^*$ acts on $  \mathcal{P}(X)$. Set 
 \[C_A:=  \sup \left\{ \frac{e^{ \mathrm{Lip}_{A,{D}} \sum_{i=1}^\infty d_i^{-\alpha} t}-1}{t} : 0 < t \leq 1  \right\} \]

\begin{lemma}[Local contraction for Dirac measures] 
\label{loc} 
Assume that $n$ is such that $(d_n)^{-\alpha} \leq \frac{3}{8}$ and that $a = \max\{8 C_A /3,1\} $. Then, for all $x, y \in X$ with  $\widetilde{D}(x,y) < 1$, 
\begin{equation} 
\label{mxnw}   
W_{\widetilde{D}} \,(\, (\mathcal{L}_A^*)^n   (\delta_x) ,                 (\mathcal{L}_A^*)^n   (\delta_y) \,)\leq \frac{3}{4}\,  W_{\widetilde{D}} \,(\delta_x, \delta_y) \nonumber \;.       
\end{equation}
\end{lemma}

\begin{proof} As $A$ is normalized, the arguments in the proof of Lemma \ref{lemma-Holder} simplify as follows. Assume that $a \geq 1$, 
$x,y \in X$ with $\widetilde{D}(x,y) < 1$ and that $\varphi \in  \mathrm{Lip}(\widetilde{D}, X)$ with $\inf(\varphi) = 0$. Then, by Lemma \ref{pio},  
\begin{align*}
& \left|\mathcal{L}^n_A(\varphi)(x) - \mathcal{L}^n_A(\varphi)(y) \right|\\
& \leq \int e^{S_n(A)\circ \tau_a(x)}\left| \varphi\circ \tau_a(x) -   \varphi\circ \tau_a(y) \right| dm_n(a) \\
& \quad +  \int \left|e^{S_n(A)\circ \tau_a(x)} -    e^{S_n(A)\circ \tau_a(y)}  \right| \left| \varphi  \circ \tau_a(y) \right| dm_n(a) \\
& \leq \mathrm{Lip}_{\varphi,\widetilde{D}} \widetilde{D}(\tau_a(x),\tau_a(y))   
\\
& \quad + \sup_a\left|  e^{S_n(A)\circ \tau_a(x)- S_n(A)\circ \tau_a(y)} -1 \right|
\int e^{S_n(A)\circ \tau_a(y)} \left| \varphi  \circ \tau_a(y) \right| dm_n(a)
\\
&
\leq \mathrm{Lip}_{\varphi,\widetilde{D}} \widetilde{D}(\tau_a(x),\tau_a(y)) 
+ \sup_a\left|  e^{S_n(A)\circ \tau_a(x)- S_n(A)\circ \tau_a(y)} -1 \right|
\|\varphi\|_\infty \\
& \leq d_n^{-\alpha} \, \mathrm{Lip}_{\varphi,\widetilde{D}} \,  \widetilde{D}(x,y) +  \left(e^{ \mathrm{Lip}_{A,{D}} \sum_{i=1}^\infty d_i^{-\alpha} D(x,y)}-1 \right)  \mathrm{Lip}_{\varphi,\widetilde{D}}  \\
& \leq   \mathrm{Lip}_{\varphi,\widetilde{D}}\left( d_n^{-\alpha}  \; \widetilde{D}(x,y) + C_A {D}(x,y) \right) = \mathrm{Lip}_{\varphi,\widetilde{D}}\left( d_n^{-\alpha}  + \frac{C_A}{a} \right) \widetilde{D}(x,y). 
\end{align*}
Observe that $ \mathcal{L}_A(1)=1$ implies that 
$$
\mathcal{L}_A(\varphi + c)(x) - \mathcal{L}_A(\varphi + c)(y) = \mathcal{L}_A(\varphi)(x) - \mathcal{L}_A(\varphi)(y)\;.
$$

Hence, by combining  Kantorovich's duality with the above estimate,
\begin{align*}
&  W_{\widetilde{D}} \left((\mathcal{L}_A^*)^n   (\delta_x) , (\mathcal{L}_A^*)^n   (\delta_y)   \right)\\ & 
= \sup \left\{ \left|\mathcal{L}^n_A(\varphi)(x) - \mathcal{L}^n_A(\varphi)(y) \right| :  \mathrm{Lip}_{\varphi,\widetilde{D}} \leq 1  \right\} \\
& \leq  \left( d_n^{-\alpha}  + \frac{C_A}{a} \right) \widetilde{D}(x,y) = \left( d_n^{-\alpha}  + \frac{C_A}{a} \right) W_{\widetilde{D}}(x,y).
\end{align*}
The result follows from this and the above choices for $a$ and $n$.
\end{proof}

We now analyze  the action of $\mathcal{L}_A^*$ on $\delta_x$
 and $\delta_y$ for $\widetilde{D}(x,y) = 1$. 
\begin{lemma}[Global contraction for Dirac measures] \label{global-contraction} If $\widetilde{D}(x,y) = 1$ and 
$n$ is chosen such that $aD(x,y) < d_n^{\alpha} $, then
\[W_{\widetilde{D}} ((\mathcal{L}_A^*)^n   (\delta_x) ,                 (\mathcal{L}_A^*)^n   (\delta_y) )\leq 
1 - e^{-\sum_{i=1}^\infty d_i^{-\alpha} D(x,y)}\left(1 -  a d_n^{-\alpha} D(x,y) \right).      
\]
\end{lemma}

\begin{proof}
Define 
\[ R_n := \int e^{\min\{ S_n(A)\circ \tau_a(x),S_n(A)\circ \tau_a(y) \} } \delta_{(\tau_a(x),\tau_a(y)) } dm_n,\]
which is a measure on $X \times X$ such that, for any Borel set $A$ and $\pi_i(x_1,x_2) := x_i$, 
\begin{align*} R_n(\pi_i^{-1} (A)) & =   \int 1_A\circ\tau_a(\pi_i(x,y))  e^{\min\{ S_n(A) \tau_a(x),S_n(A) \tau_a(y) \} } dm_n \\
& \leq  \mathcal{L}_A^n(1_A)(\pi_i(x,y))  =    (\mathcal{L}_A^*)^n   (\delta_{\pi_i(x,y)})(A). 
\end{align*}
As it is well known from the theory of couplings, this estimate implies that $R_n$ can be extended to an element in $\Gamma( (\mathcal{L}_A^*)^n   (\delta_{x}), (\mathcal{L}_A^*)^n   (\delta_{y}))$. Or, by a straightforward calculation, one might check that the measure 
\[
Q_n := R_n + \frac{(  (\mathcal{L}_A^*)^n   (\delta_{x})  - R_n \circ \pi_1^{-1}) \otimes   (  (\mathcal{L}_A^*)^n   (\delta_{y})  - R_n \circ \pi_2^{-1}) }{1- R_n(X\times X)}
\]
in fact is such a coupling.
Now observe that $|S_n(A)\circ \tau_a(x)(x) - S_n(A)\circ \tau_a(y)| \leq C D(x,y)$ for $C:= \sum_{i=1}^\infty d_i^{-\alpha}$. Hence, for 
$$\Delta_\epsilon := \{ (x,y) : \widetilde{D}(x,y) \leq \epsilon \}$$ 
and $\epsilon = d_n^{-\alpha} D(x,y)$, we have that 
\begin{align*}
Q_n(\Delta_\epsilon) &
\geq R_n(\Delta_\epsilon) \\
&\geq  \int e^{\min\{ S_n(A)\circ \tau_a(x),S_n(A) \circ \tau_a(y) \} } dm_n \\
& \geq  \int e^{ S_n(A)\circ \tau_a(x)} dm_n \ e^{-C D(x,y)}.
 \end{align*}
Hence, if $n$ is chosen such that $\epsilon = d_n^{-\alpha} D(x,y) < 1/a$, 
\begin{align*}
& W_{\widetilde{D}} ( (\mathcal{L}_A^*)^n   (\delta_x) ,                 (\mathcal{L}_A^*)^n   (\delta_y) ) \\
& \leq \int \widetilde{D}(x,y) dQ_n(x,y) \\ 
&= \int_{\Delta_\epsilon^c}   \widetilde{D}(x,y) dQ_n(x,y) +  \int_{\Delta_\epsilon}   \widetilde{D}(x,y) dQ_n(x,y)\\
& \leq 1 - Q_n(\Delta_\epsilon) + a \epsilon Q_n(\Delta_\epsilon) \\
&= 1 - Q_n(\Delta_\epsilon)(1 - a \epsilon) \\
& \leq 1 - e^{-C D(x,y)}(1 - a d_n^{-\alpha} D(x,y)).
 \end{align*}
The assertion follows when including the case $\epsilon \geq 1/a$ for completeness.
\end{proof}

In order to employ Lemma \ref{loc} and \ref{global-contraction} for getting estimates for arbitrary measures, we recall the following basic fact:
assume that $n \in \mathbb{N}$ and  $Q_{\xi,\eta} \in \Gamma( (\mathcal{L}_A^*)^n (\delta_\xi), (\mathcal{L}_A^*)^n (\delta_\eta) )$, are such that, $(\xi, \eta) \mapsto \int \widetilde{D}(x,y) dQ_{\xi, \eta}(x,y)$ is measurable, then, for any $\mu,\nu \in \mathcal{P}(X)$ and any $\Pi \in \Gamma(\mu,\nu)$, we  get that 
 \[ dQ_{\xi,\eta}(x,y) d\Pi(\xi,\eta) \in   \Gamma( (\mathcal{L}_A^*)^n (\mu) , (\mathcal{L}_A^*)^n (\nu) ). \]
And, in particular, by taking the infimum over all possible couplings $Q_{\xi,\eta}$,  
\begin{equation} 
\label{rre2} 
W_{\widetilde{D}} (( \mathcal{L}_A^*)^n   (\mu) ,                 (\mathcal{L}_A^*)^n (\nu) ) \leq    \int   W_{\widetilde{D}} \,(\mathcal{L}_A^*)^n (\delta_x), (\mathcal{L}_A^*)^n(\delta_y))\, d \Pi(x,y).
 \end{equation}

This allows us to  summarize Lemma \ref{loc} and \ref{global-contraction} into a single estimate. In order to do so, set
\[ r_n(t) := 
\begin{cases}
1 - e^{-\sum_{i=1}^\infty d_i^{-\alpha} t}\left(1 -  \min \{ 1, t a d_n^{-\alpha} \} \right) & : \; t \geq \frac{1}{a} \\
\frac{3}{4} &: \; t < \frac{1}{a}
\end{cases}
\]
\begin{lemma}\label{total} Assume that $\mu,\nu \in \mathcal{P}(X)$ and  that $\Pi \in \Gamma(\mu,\nu)$. Then, for any $n$ such that $(d_n)^\alpha \geq 8/3$, 
\[ W_{\widetilde{D}} ( \mathcal{L}_A^*)^n   (\mu), (\mathcal{L}_A^*)^n (\nu) ) \leq  \int r_n(D(x,y))  \widetilde{D}(x,y) d  \Pi(x,y). 
\] 
\end{lemma}
 
\begin{proof} By dividing the domain of integration in \eqref{rre2} into $\{ (x,y): D(x,y) < 1/a \}$,  $\{(x,y): 1/a \leq D(x,y) < d_n^\alpha/a\}$ and  $\{(x,y): D(x,y) \geq d_n^\alpha/a\}$, the first immediately follows by application of 
 Lemma \ref{loc} and \ref{global-contraction}.   
\end{proof} 

As an application, we obtain uniform contraction on certain subsets of $\mathcal{P}(X)$ which we define now. For $\gamma > 0$ and $\epsilon > 0$, set 
\[ \mathcal{P}_{\gamma,\epsilon} (X) := \left\{ \mu \in \mathcal{P}(X) : \mu(\{ x \in X: D(x,0) > \gamma \})< \epsilon  \right\}. \]
\begin{proposition}\label{prop:global-with-tails} Assume that $\mu,\nu \in \mathcal{P}(X)$ and that $\gamma \geq  1/a$ 
is chosen such that $\mu,\nu \in \mathcal{P}_{\gamma,\epsilon} (X)$, for $\epsilon := W_{\widetilde{D}}(\mu,\nu)/4$. Then, for any   $n$ with $(d_n)^\alpha \geq 8/3$,   
\[ W_{\widetilde{D}} ((	 \mathcal{L}_A^*)^n   (\mu), (\mathcal{L}_A^*)^n (\nu) )  \leq 
\left( 
1 - \frac{1 -  \min \{ 1,  2 a \gamma d_n^{-\alpha} \} }{2e^{2 \gamma \sum_{i=1}^\infty d_i^{-\alpha} }}
\right)
W_{\widetilde{D}}(\mu,\nu).
\]
\end{proposition}

\begin{proof} By the triangle property, $D(x,y) \leq 2 \gamma$ whenever $D(x,0),D(y,0) \leq  \gamma$. Hence, 
\[\{ D(x,y) > 2\gamma \} \subset (\{ x: D(x,0) > \gamma\} \times X) \cup (X \times \{ y: D(0,y) > \gamma\}).  \]
Therefore, for any $\Pi \in \Gamma(\mu,\nu)$, $\Pi(\{ D(x,y) > 2\gamma \}) < 2 \epsilon$.  Now assume that $\Pi$ is an optimal coupling, that is, $W_{\widetilde{D}}(\mu,\nu) = \int \widetilde{D} d\Pi$, and that $\gamma > 1/a$. Then, by Lemma \ref{total}, 
\begin{align*}  
& W_{\widetilde{D}} ( \mathcal{L}_A^*)^n   (\mu), (\mathcal{L}_A^*)^n (\nu) ) \\  
&\leq \int \left(1_{\{D(x,y) \leq 2\gamma\}} r_n(2\gamma) + 1_{\{D(x,y) > 2\gamma \}}\right) \widetilde{D}(x,y) d\Pi\\
 & =  r_n(2\gamma) \int \widetilde{D}(x,y) d\Pi + (1- r_n(2\gamma)) \Pi(\{D(x,y) > 2\gamma \}) \\
 & \leq \left( r_n(2\gamma)  + \frac{2\epsilon (1-r_n(2\gamma))}{W_{\widetilde{D}}(\mu,\nu)} \right) W_{\widetilde{D}}(\mu,\nu) \\
 &  =  \left( 1 - \frac{1 - r_n(2\gamma)}{W_{\widetilde{D}}(\mu,\nu)}(W_{\widetilde{D}}(\mu,\nu) - 2\epsilon ) \right) W_{\widetilde{D}}(\mu,\nu) \\
 &= \left( 1 - \frac{1 - r_n(2\gamma)}{2}  \right) W_{\widetilde{D}}(\mu,\nu).
\end{align*}
This proves the assertion.  
\end{proof}

As an immediate consequence, one obtains that for any $\mu,\nu \in \mathcal{P}(X)$, there exists $n$ such that the distance between $(\mathcal{L}_A^*)^n (\mu)$ and $(\mathcal{L}_A^*)^n (\nu)$ is strictly contracted, even though the contraction rate depends on the distance as well as the tails of $\mu$ and $\nu$. 
In particular, $L = \lim W_{\widetilde{D}}((\mathcal{L}_A^*)^n (\mu),(\mathcal{L}_A^*)^n (\nu))$ exists. In order to guarantee that $L=0$ and that $((\mathcal{L}_A^*)^n (\mu))_n$ converges, we now employ the notion of adapted tails as introduced in 
Definition \ref{Def:adapted}.
Recall that 
we say that 
$m$ has adapted tails if, for any $\epsilon > 0$ there exists a sequence of positive numbers $(\kappa_n)$ such that 
\begin{enumerate}
\item $\sum_{n=1}^\infty m(\mathbb{R}\setminus[-\beta_1^n \kappa_n,\beta_1^n \kappa_n]) < \epsilon$,
\item the sequence $ (\kappa_n) $ is in $X$. 
\end{enumerate}
As shown below, this condition allows to control the tails of $(\mathcal{L}_A^*)^n (\mu)$ asymptotically. Before giving examples, we prove the main result of this appendix. 
\begin{theorem}\label{theo:conformal-measure} 
Assume that $X = c_0(\mathbb{R})$ or $X= l^p(\mathbb{R}),\; 1 \leq p < \infty$  and that $L : X \to X$ a weighted shift satisfying \eqref{d}. Moreover assume that $A \in  \mathrm{Lip}_b(D_X^{\alpha}, X)$ with $\mathcal{L}_A(1) =1$ and that $m$ has adapted tails. Then, for any $\nu \in  \mathcal{P}(X)$, the limit 
\[ \mu := \lim_{n \to \infty}  (\mathcal{L}_A^*)^n (\nu) \]
exists in $\mathcal{P}(X)$. 
Moreover, $\mu$ is the unique element in $\mathcal{P}(X)$ such that  
$$\mathcal{L}_A^* (\mu) = \mu.$$
\end{theorem}

\begin{proof} As the first step in the proof, we derive a formula for the asymptotic tail  of $ (\mathcal{L}_A^*)^n (\nu)$. 
In order to do so, 
for $ (\kappa_j)  \in X$,   $n \in \mathbb{N}$ and $K > 0$ given, set 
\[ \Omega 
 := \left\{ (x_j) : x_j \in [-\kappa_j,\kappa_j ] \hbox{ for } 0 \leq j < n, \; \|(x_n, x_{n+1}, \ldots)   \|_X \leq K\right\}. 
  \]   
Then, for $\nu \in  \mathcal{P}(X)$ and $\beta_1^0 :=1$, 
\begin{align*} 
 &  (\mathcal{L}_A^*)^n (\nu)\left(\Omega^c\right) \\
 & = (\mathcal{L}_A^*)^n (\nu)\left(\textstyle \bigcup_{j=0}^{n-1} \{(x_i) : |x_j| > \kappa_j \} \cup
 \{ (x_i) : \|(x_n, x_{n+1}, \ldots)   \|_X > K \} \right) \\
 & \leq \sum_{j=0}^{n-1}  (\mathcal{L}_A^*)^n (\nu)\left(  \{|x_j| > \kappa_j \} \right) + 
 (\mathcal{L}_A^*)^n (\nu)\left(  \{ \|(x_n, x_{n+1}, \ldots)   \|_X > K \} \right)\\
& =  \int \mathcal{L}_A^n \left( \sum_{j=0}^{n-1} 
1_{ \left\{ {|x_0|}/{\beta_1^j} > \kappa_j \right\} }\circ L^{j}  + 
1_{ \left\{ \|( {x_n}/{\beta_1^n},  {x_{n+1}}/{ \beta_2^n}, \ldots)   \|_X > K   \right\}  }\circ L^{n} \right) d\nu \\
& =   \int \mathcal{L}_A^{n-j} \left( 
1_{ \{| {x_0}| > {\beta_1^j}\kappa_j \} } \right) d\nu
+ \int   
1_{ \{ \|( \frac{x_n}{\beta_1^n}, \frac{x_{n+1}}{ \beta_2^n}, \ldots)   \|_X > K \}   }  d\nu \\
& \leq  \sum_{j=0}^{n-1} \int \mathcal{L}_A^{n-j} \left( 
1_{ \{| {x_0}| > {\beta_1^j}\kappa_j \} } \right) d\nu
+
\nu  \left( 
  \{ (x_i) : \|(x_i) \| > d_n K  \}  \right). 
\end{align*}
As $A$ is bounded, it follows that $e^A = C^{\pm 1}$. In particular, $\mathcal{L}_A(1_B) = C^{\pm 1} m(B)$ for any Borel set $B$ which only depends on the first coordinate.      
Hence, if $(\kappa_j)$ is chosen for $\epsilon > 0$ according to Definition \ref{Def:adapted}, then, as $A$ is normalized,  
\begin{align} \nonumber
  (\mathcal{L}_A^*)^n (\nu)\left(\Omega^c\right)  & \leq 
  C \sum_{k=0}^{n-1} m(\{x \in \mathbb{R} : |x| > {\beta_1^j}\kappa_j  \} ) 
+
\nu  \left( 
  \{ (x_i) : \|(x_i) \| > d_n K  \}  \right)\\
\label{eq:tails}  
  & \leq C \epsilon + \nu  \left( 
  \{ (x_i) : \|(x_i) \| > d_n K  \}  \right) \xrightarrow{n \to \infty } C\epsilon. 
\end{align}

In the second step, we show that $ W_{\widetilde{D}} ((\mathcal{L}_A^*)^n(\mu),(\mathcal{L}_A^*)^n(\nu)) \to 0$ as $n \to \infty$.
So assume that $\mu,\nu \in \mathcal{P}(X)$. By Lemma \ref{total}, $W_{\widetilde{D}} ((\mathcal{L}_A^*)^n(\mu),(\mathcal{L}_A^*)^n(\nu))$ is decreasing in $n$ and, in particular, 
\[L := \lim_{n\to \infty} W_{\widetilde{D}} ((\mathcal{L}_A^*)^n(\mu),(\mathcal{L}_A^*)^n(\nu))\] exists. Now assume that $L>0$. Then, as $m$ is adapted,  there 
exists $(\kappa_j) \in X$ with   $\sum_{j=1}^\infty m([-\beta_1^j \kappa_j,\beta_1^j \kappa_j]^{c}) < L/(8C)$. For $K:= \|(\kappa_j)\|_X$,    
the estimate  \eqref{eq:tails} combined with $d_n \to \infty$ implies that there exists $n_0 \in \mathbb{N}$ such that
 $ (\mathcal{L}_A^*)^n (\nu)\left(\Omega^c\right) < L/4 $ for any $n \geq n_0$. However, as $\Omega$ for these choices of $n$ and $K$ satisfies $\Omega \subset \{ x \in X: \|x\|_X \leq 2\|(\kappa_i)\|_X \}$, it follows that 
\[(\mathcal{L}_A^*)^n(\mu),(\mathcal{L}_A^*)^n(\nu) \in  \mathcal{P}_{2\|(\kappa_i)\|_X, L/4}(X) \hbox{ for all } n \geq n_0.\]   
We now apply Proposition \ref{prop:global-with-tails} to the pair $((\mathcal{L}_A^*)^{n}(\mu),(\mathcal{L}_A^*)^{n}(\nu))$. 
Namely, for $k$ chosen such that the factor on the right hand side in Proposition  \ref{prop:global-with-tails} is strictly smaller than 1, it follows that 
\[L = \lim_{n\to \infty} W_{\widetilde{D}} ((\mathcal{L}_A^*)^{n+k}(\mu),(\mathcal{L}_A^*)^{n+k}(\nu)) < \lim_{n\to \infty} W_{\widetilde{D}} ((\mathcal{L}_A^*)^{n}(\mu),(\mathcal{L}_A^*)^{n}(\nu)) =L,\]
which is absurd. Hence, $L=0$ and the assertion of the second step is proven.

In the third step, we  use \eqref{eq:tails} to obtain that $(\mathcal{L}_A^*)^n (\delta_0)$ is tight, where $0$ is the origin in $X$. 
We therefore recall the notion of tightness and Prokhorov's theorem. Namely, a sequence $(\mu_n : n \in \mathbb N)$ of probability measures is tight if for any $\epsilon >0$ there exists a compact set $K$ such that $\mu_n (K) > 1- \epsilon$ for any $n \in \mathbb N$. Prokhorov's theorem then states that tightness and sequential compactness are equivalent.   

Now assume that $\epsilon > 0$ and that   $(\kappa_j)$ is chosen  according to Definition \ref{Def:adapted}, and set $\Omega_\epsilon := \{(x_i) \in X : |x_i|\leq \kappa_i \; \forall i \geq 0 \}$. 
As $0 \in \Omega_\epsilon$  the estimate \eqref{eq:tails} implies that
\[ (\mathcal{L}_A^*)^n (\delta_0) \left( \Omega_\epsilon \right) \leq C\epsilon. \]
Moreover, as  it easily  can be shown, $\Omega_\epsilon$ is compact and therefore, $(\mathcal{L}_A^*)^n (\delta_0)$ is tight. Hence, by Prokhorov's theorem, there exists a subsequence $(n_k)$ with $n_k \nearrow \infty$ such that $ \mu :=\lim_{k \to \infty} (\mathcal{L}_A^*)^n (\delta_0) $ exists. 
   
In the last step, we prove  for any $\nu \in  \mathcal{P}(X)$ that $ \mu = \lim_{n \to \infty}  (\mathcal{L}_A^*)^n (\nu)$ and that $\mu = \mathcal{L}_A^*(\mu)$. 
In order to do so, we apply Step 2 to the measures $\delta_0$ and $\mathcal{L}_A^* (\delta_0)$, which implies by continuity of  $\mathcal{L}_A^*$ that
\[\mu = \lim_{k \to \infty} (\mathcal{L}_A^*)^n (\delta_0) = \lim_{k \to \infty} (\mathcal{L}_A^*)^{n+1} (\delta_0)  = \mathcal{L}_A^*(\mu).\] 
Finally, the remaining assertion follows from a further application of Step 2 to $\nu$ and $\mu$.
\end{proof}

We now  provide some examples of adapted measures. Therefore recall that a measure on $\mathbb{R}$ has polynomial tails of order $\gamma$ if there exist $C >0$ and  $\gamma > 1$ such that 
\[ m(\{ x : |x| > z \}) \leq C z^{-\gamma} \hbox{ for all } z>0.  \]   
Hence,  for class of measures and  $\kappa_n := Bn^\ell/d_n$, for some $B> 0$ and $\ell> \gamma^{-1}$, it follows that  
 \begin{align*}
\sum_{n=1}^\infty m(\{ x : |x| >  \kappa_n \beta_1^n \}) 
& 
 \leq  \sum_{n=1}^\infty m(\{ x : |x| >  \kappa_n d_n\}) 
 \leq  C \sum_{n=1}^\infty L^{-\gamma} n^{-\ell \gamma}    \\
& 
< \frac{C\ell \gamma}{\ell \gamma -1} L^{-\ell \gamma} \xrightarrow{L \to \infty} 0. 
\end{align*}
Hence, $m$ is adapted whenever $(n^\ell/d_n) \in X$ for some $\gamma^{-1} < \ell \leq 1$. The following slightly more specific examples of  adapted measures $m$ with polynomial tails easily follows from this observation.
\begin{proposition} \label{prop:polynomial-tails}
Assume that $m$  has polynomial tails of order $\gamma$. Then the following holds.
\begin{enumerate}
\item If $n/d_n  \to 0$ for some $\ell > \gamma^{-1}$  and $X = c_0(\mathbb{R})$, then   $m$  is adapted.
\item If $d_n > Cn^{\ell} $ for some $\ell > \gamma^{-1} + 1/p$, and $X= l^p(\mathbb{R})$ with  $1 \leq p < \infty$,
then   $m$  is adapted.
\end{enumerate} 
In particular, if  $d_n$ grows exponentially, then any $m$  with polynomial tails is adapted.
\end{proposition}
Moreover, we say that a measure on $\mathbb{R}$ has exponential tails if there exist $C >0$ and  $0 < \gamma < 1$ such that 
\[ m(\{ x : |x| > z \}) \leq C {\gamma}^z \hbox{ for all } z>0.  \]   
By a similar argument as above, on obtains the following criterion.
\begin{proposition}  \label{prop:exponential-tails}
Assume that $m$  has exponential tails and that $((\log n)/d_n) \in X$. Then  $m$  is adapted.
\end{proposition}

\end{document}